%% file: mbskmk_13-11-16.tex
\documentclass[a4paper,10pt,reqno]{amsart}
\usepackage{graphicx}
\usepackage[latin1]{inputenc}
\usepackage{amsmath,amsthm,amssymb,amsfonts}
\usepackage{latexsym}
\usepackage{color}     
\input{13macros.tex}
\numberwithin{equation}{section}

\newcommand{\average}{{\mathchoice {\kern1ex\vcenter{\hrule
height.4pt width 8pt depth0pt}
\kern-11pt} {\kern1ex\vcenter{\hrule height.4pt width 4.3pt
depth0pt} \kern-7pt} {} {} }}

\newcommand{\wstar}{{\stackrel{*}{\rightharpoonup}}\,}

\newcommand{\calA}     {\mathcal{A}}

\newcommand{\calH}     {\mathcal{H}}

\newcommand{\R}         {\mathbb{R}}
\newcommand{\N}         {\mathbb{N}}

\newcommand{\QSLB}         {{\rm QSLB}}
\newcommand{\GYM}         {{\rm GYM}}

\newcommand{\genYM} 				{Y_{\rm gen}}
\newcommand{\DM}            {DM}
\newcommand{\id}            {{\rm id}}

\let\e=\varepsilon
\let\O=\Omega
\let\la=\langle
\let\ra=\rangle

\def\debaixodaseta#1#2{\mathrel{}\mathop{\longrightarrow}\limits^{#1}_{#2}}
\def\debaixodasetafraca#1#2{\mathrel{}\mathop{\rightharpoonup}\limits^{#1}_{#2}}

\newcommand{\rca}{\mathcal{M}}
\newcommand{\prca}{\mathcal{M}_{\vspace*{.5mm}1}^+}
\newcommand{\md}{{\rm d}}
\renewcommand{\d}{{\rm d}}
\newcommand{\ups}{\Upsilon_\mathcal{S}(\R^{M\times N})}

\newcommand{\us}{{\Upsilon}}

\begin{document}
 \title{Generalized  $\mathbf{W^{1,1}}$-Young measures and relaxation of  problems with linear growth}

\author{Margarida Baia}\address{MB: Departamento de Matem\'{a}tica Instituto Superior T\'{e}cnico
Av. Rovisco Pais 1049-001 Lisbon, Portugal, email:mbaia@math.tecnico.ulisboa.pt }  
\author{Stefan Kr\"{o}mer}\address{SK: Math.~Inst., Universit\"{a}t zu K\"{o}ln, 50923 K\"{o}ln, Germany, email: kroest@gmx.de} 
\author{Martin Kru\v{z}\'{\i}k}\address{MK: Institute of Information Theory and Automation of the CAS, Pod vod\'{a}renskou
v\v{e}\v{z}\'{\i}~4, CZ-182~08~Praha~8, Czech Republic and
Faculty of Civil Engineering, Czech Technical
University, Th\'{a}kurova 7, CZ-166~ 29~Praha~6, Czech Republic, email:  kruzik@utia.cas.cz}

\begin{abstract}
We completely characterize generalized Young measures generated by sequences of gradients of maps from $W^{1,1}(\O;\R^M)$  where  $\O\subset\R^N$. This extends and completes 
previous analysis by Kristensen and Rindler \cite{KriRin_YM_10} where concentrations of the sequence of gradients at the boundary of $\O$ were excluded. We apply our results to relaxation of non-quasiconvex variational problems with linear growth at infinity. We also link our characterization to Sou\v{c}ek spaces \cite{soucek}, an extension of $W^{1,1}(\O;\R^M)$ where gradients are considered as measures on $\bar\O$.  
\end{abstract}


\maketitle

%*************************************************\\
%TODO:
%\begin{itemize}
%\item polish introduction (in particular more on relaxation, see below)
%\item relaxation of non-quasiconvex variational problems with linear growth at infinity (or kill that from  abstract and introduction)
%\item {}[OPTIONAL] maybe reduce the local decomposition lemma and related stuff to the special case $J=2$, $K_1=\partial \O$, $K_2=\bar\O$, which is the only case we use
%\item Necessary and sufficient conditions on the boundary (Sec 6,7) are complete, but maybe not so easy to understand for outsiders... Especially Sec 6 could use a little more structure/some introduction.
%\end{itemize}
%STUFF TO KEEP IN MIND FOR PROOFREADING:
%\begin{itemize}
%\item inconsistent notation for integrals w.r.t. non-Lebesgue measures. Write $\int \lambda(\md x)$ or $\int \md \lambda(x)$, but only one of the two throughout. (SK: mostly my fault, sorry... Anyway, I really prefer the latter here, because especially for singular BV derivatives, i.e., $\lambda=|Du^s|$, the former looks quite strange to me)
%\item check that the notation for unit half-balls consistently is $D_\varrho$ (not $D_\varrho$, as (partially) in an earlier version)
%\item check and fix citations
%\item check and fix broken {\textbackslash}ref commands (some were brought here by copy/paste, but not adapted)
%\item Appendix B: check for confusion $\calR\leftrightarrow\calS$
%\end{itemize}
%END TODO\\
%*************************************************\\

\section{Introduction}
Oscillations and concentrations appear in many problems of  the
calculus of variations,  partial differential equations, and/or
optimal control theory; cf. \cite{bk} for an overview.  While  Young measures \cite{y} successfully capture oscillatory behavior  of sequences, they completely miss
concentration effects. These concentrations effects may be dealt with appropriate generalizations of Young measures, as in DiPerna's and Majda's treatment of
concentrations \cite{diperna-majda}, following
 Alibert's and Bouchitt\'{e}'s approach
in  \cite{ab} (see also \cite{fonseca}), etc. Detailed overviews of this subject may be found in \cite{r,tartar1}. A general feature of all these approaches 
is the description of oscillations and/or concentrations in sequences in terms of parametrized measures.   
Besides the work of Kinderleherer and Pedregal \cite {KP_91, KP_94} dealing with Young measures, and thus with oscillations only, both concentrations and oscillations were jointly treated in \cite{KK2011} in $W^{1,p}$ for $p>1$ and in \cite{KriRin_YM_10,rindler} in $W^{1,1}$ and $\mathrm{BV}$. Applications of parametrized measures in optimal control of dynamical systems with linear growth  involving oscillations and concentrations  
can be found e.g.\ in \cite{chk,k-r-control,r}.

Explicit characterization of such parametrized measures is very important for applications because it allows us to analyze limit problems.  The parametrized measures generated by sequences of gradients are of a particular interest in vectorial multidimensional calculus of variations, for instance when we want to minimize integral functionals $I(u):=\int_\O v(\nabla u(x))\,\md x$ with a non-convex (or non-quasiconvex) integrand $v$ having $p$-growth at infinity, i.e., for large arguments. The functional $F$ is then not  lower semicontinuous with respect to weak convergence in $W^{1,p}(\O;\R^m)$ and its minimizers do not necessarily exist. Parametrized measures are then applied to define the so-called relaxed problem tracing the limit behavior of minimizing sequences.  For sequences of gradients of maps living  in $W^{1,p}(\O;\R^M)$, Kinderlehrer and Pedregal characterized Young measures in \cite{KP_91, KP_94} for $p=+\infty$ and $1<p<+\infty$, respectively. If $p>1$, concentrations effects in minimizing sequence of coercive functionals are excluded using the so-called Decomposition Lemma proved in \cite{fmp} which allows to construct minimizing sequences which do not exhibit concentrations.  Consequently, Young measures suffice to describe the relaxed functional.  On the other hand, concentration effects can appear if $v$ is non-coercive, more precisely if 
$|v(A)|\le C(1+|A|^p)$ for some $C>0$. Generalizations of Young measures were used in \cite{kkk,mkak,kruzik} to analyze  necessary and sufficient conditions guaranteeing weak lower semicontinuity of $F$; cf. also \cite{bk}.

If $p=1$, i.e, for functionals with linear growth, concentration effects are usually always relevant and often prevent the existence of minimizers in $W^{1,1}$, a well-know fact intimately related to the lack of reflexivity of $W^{1,1}$. For precisely that reason, the standard approach to handle functionals with linear growth relies on extension to $\mathrm{BV}$, which in particular allows to keep track of certain concentrations of the gradient in the interior of the domain, as singular contributions of the derivative in $\mathrm{BV}$. Moreover, if the integrand lacks (quasi)convexity, 
oscillations and concentrations appear simultaneously and are mutually interconnected. The methods used to characterize Young measures and their generalizations for  $p>1$ cannot then be applied; in particular, the Decomposition Lemma \cite[Lemma 1.2]{fmp}, which allows us to replace a sequence of gradients  generating a Young measure by a new sequence of gradients such that the $p$-th power of the norm in equiintegrable,  does not hold in $W^{1,1}$ or $\mathrm{BV}$. 

In \cite{KriRin_YM_10}, the authors resolved this problem 
%for integrands which posses a (positively one-homogeneous) recession function and if 
\emph{assuming} that
concentrations do not appear at the boundary of the domain $\O$.
As also shown in \cite{KriRin_YM_10}, in particular this is enough to discuss weak$^*$ lower semicontinuity for integral functionals in $BV$
subject to a Dirichlet condition, because it is then possible to extend both the integrand and the admissible functions to a larger domain $\tilde\O$, where everything will be fixed near the boundary. This of course excludes concentration effects near the boundary of $\partial\O$.
On the other hand, as soon as there is at least a part of the boundary where we do not impose a Dirichlet condition,
one cannot fully avoid this phenomenon as the following 
simple example illustrates.

Consider the following one-dimensional variational problem:
\begin{align}\label{toyproblem}
\min_{u\in W^{1,1}(0,1)}\left\{I(u):=\int_0^1 ((x-1)^2+\varepsilon)|u'(x)|\,\md x+ u(0)^2+ (u(1)-1)^2\right\}
\end{align}
for given $0<\varepsilon<1.$
Notice that although no Dirichlet boundary condition was imposed, we still get the following natural 
boundary conditions for critical points, at least formally:
\[
\begin{alignedat}{2}
& (1+\varepsilon)\frac{u'}{\abs{u'}}\varrho+2u=0 \,\,\, && \text{at $x=0$} \,\,\,\, \text{and} \,\,\,\,	&&&\varepsilon\frac{u'}{\abs{u'}}\varrho+2(u-1)=0 \,\,\, &&&&& \text{at $x=1$ (nonlinear Robin condition),}
\end{alignedat}
\]
where $\varrho$ is the outer normal at the boundary: $\varrho(0)=-1$, $\varrho(1)=+1$. 
Roughly speaking, members of a minimizing sequence  $(u_n)_{n\in\N}$ for \eqref{toyproblem} will try to keep their derivative as small as possible but also need to achieve some values close to zero and one at $x=0$ and $x=1$, respectively. Due to the weight $x\mapsto(x-1)^2+\varepsilon$ the best place to switch from one constant   to another one is near $x=1$. In fact, the infimum of $I$ can be calculated explicitly, and we can also give an explicit minimizing sequence:
For every $u$ admissible, 
$$
	I(u)> \varepsilon\int_0^1|u'(x)|\,\md x+u(0)^2+(u(1)-1)^2\ge \varepsilon (u(1)-u(0))+u(0)^2+ (u(1)-1)^2\ .
$$
The function $(u(0),u(1))\mapsto\varepsilon (u(1)-u(0))+u(0)^2+ (u(1)-1)^2$ is minimized  for $u(0)=\varepsilon/2$ and  $u(1)=(2-\varepsilon)/2$ with the value $(2\varepsilon-\varepsilon^2)/2$. 
On the other hand, take
$$
u_n(x):=\begin{cases}
\varepsilon/2 &\text{ if $0\le x\le 1-1/n$}\\
n(1-\varepsilon)x+\varepsilon/2-(1-\varepsilon)(n-1) &\text{ otherwise.}
\end{cases}
$$
We see that $I(u_n)\to (2\varepsilon-\varepsilon^2)/2$ as $n\to\infty$, so it is a minimizing sequence. However,
$u_n\to 0 $ in $L^1(0,1)$ and $(u'_n)$ concentrates at $x=1$. 
In particular, no minimizer exists in the admissible class of competitors in $W^{1,1}(0,1)$. 

It is not entirely obvious how to relax $I$ in $\mathrm{BV}$, or other functionals (formally) involving a natural boundary condition of Robin type. In fact, if we simply replace  the integral part of $I$ with its relaxation with respect to weak$^*$ convergence in $\mathrm{BV}$ while leaving the boundary term alone,
we get for all $u\in BV(0,1)$
\begin{align*}
  I_1(u):=\int_0^1 ((x-1)^2+\varepsilon)\,\md|Du(x)|+ u(0)^2+(u(1)-1)^2.
\end{align*}
where $u(0)$ and $u(1)$ are understood in the sense of trace in $\mathrm{BV}$. However, this clearly is not the right choice, because $I_1$ is still not lower semicontinuous with respect to weak$^*$ convergence in $BV$ for $\varepsilon<2$: with the minimizing sequence $(u_n)$ defined above,
we have $u_n\rightharpoonup^* 0$ in $BV$ and $I_1(u_n)=I(u_n)\to (4\varepsilon-\varepsilon^2)/4<I_1(0)=1$.
Now, the main issue is that the trace at $x=1$ is not continuous along $(u_n)$, which here causes the term 
$(u_n(1)-1)^2$ to jump up in the limit, reflecting the fact that the concentration of $(u_n')$ at $x=1$ is being dropped in the weak$^*$-limit in $BV$. The natural way to solve this dilemma is to consider a slightly larger space, and a good choice is $X:=BV(0,1)\times \RR^2$, 
the closure of $W^{1,1}$ (and $BV$), embedded into $X$ by $u\in W^{1,1}\mapsto (u,(u(0),u(1)))\in X$,
with respect
to weak$^*$-convergence in $X$. 
%(i.e, $u_n\rightharpoonup^* u$ in $BV$, $\beta_n\to \beta$ in $\RR$.)
Essentially, we now keep track of the ``true'' boundary value at $x=0,1$ along sequences in $W^{1,1}$ possibly different from the trace of the limit function in $BV$, and it is not difficult to check that the relaxation of $I$ (and $I_1$) in $X$ is given for all $(u,(\beta_0,\beta_1))\in BV(0,1)\times \RR^2$ by 
\begin{align*}
  I_2(u,(\beta_0,\beta_1)):=\int_0^1 ((x-1)^2+\varepsilon)\,\md|Du(x)|+
	(1+\epsilon)\abs{u(0)-\beta_0}+{\beta_0}^2+\varepsilon \abs{\beta_1-u(1)}+ (\beta_1-1)^2,
\end{align*}
where $u(0)$ and $u(1)$ are understood in the sense of traces in $BV$.
From this, we can deduce the correct relaxation of $I$ in $BV$: $I_3(u):=\inf_{(\beta_0,\beta_1)\in\RR^2} I_2(u,(\beta_0,\beta_1))$.
However, this is far less natural than $I_2$, and essential information encoded by $\beta_1$ at $x=1$ is lost that way.

The example can be easily generalized to higher dimensions by considering $\O\subset\R^N$ a bounded domain 
of class $C^1$ with two disjoint parts of the boundary $\Gamma_0$ and $\Gamma_1$ of positive $N-1$ dimensional  Hausdorff  measure. For instance, for $u\in W^{1,1}(\O;\R^M)\ \& \  u(0)=0$ on $\Gamma_0$, a rough analogue of $I$ (with a slightly modified boundary term to ensure $J<\infty$ on $W^{1,1}$) is given by

$$\text{minimize } J(u):=\int_\O ({\rm dist}^2(x,\Gamma_1)+\eps) |\nabla u(x)| \,\md x+\int_{\Gamma_1} \sqrt{1+(u(x)-\bar u(x))^2}\,\md \mathcal{H}^{N-1}(x)\ ,$$
where $\Omega\subset \R^N$ is a bounded domain with a $C^1$ boundary , $\Gamma_1\subset \partial \Omega$ and $\bar u$ is a given function in 
$L^1(\Gamma_1,\mathcal{H}^{N-1}\lfloor\partial \O ; \R^M)$. Above, ``dist'' denotes the distance function.

Concerning relaxation of this kind of functional, a natural replacement for the space $X$ above is the Sou\v{c}ek space $W_\mu^1(\bar\O)\cong BV(\O)\times \cM(\partial\O)$, keeping track of the ``true'' boundary value as a measure on $\partial\O$ (the so-called outer trace in $W_\mu^1(\bar\O)$). For more details on $W_\mu^1(\bar\O)$ we refer to Section~\ref{sec:soucek}.

The previous example shows that concentration effects at the boundary can be  quite relevant 
for minimization problems with linear growth at infinity, even if all terms of the functional are convex, non-negative and coercive. Once oscillations favored by nonconvex integrands come into play on top of that,
we are lead to study the generalized gradient Young measures appearing as the natural limiting object capturing this behavior. 

The main goal of our work is to complete the characterization of generalized Young measures 
given in \cite{KriRin_YM_10}
by deriving necessary and sufficient conditions which must be 
satisfied at the boundary of the domain $\O$. 
In particular, our main result, summarized in Theorem~\ref{thm:main}, allows us to easily set up and prove relaxation results for a large class problems like our example, first in the space of generalized gradient Young measures and then in the Sou\v{c}ek space, as a straightforward corollary.

The plan of the paper is as follows. We first introduce necessary notation and collect some auxiliary results in Section~\ref{sec:preliminaries}. Then we state the main result, i.e., characterization of generalized gradient measures up to the boundary; cf.~Section~\ref{sec:main}. Proofs of necessity and sufficiency of our new condition
at the boundary are the subject of Sections~\ref{sec:nec} and \ref{sec:suff}, respectively.  The first one, in Section~\ref{sec:soucek}, is a relationship of generalized Young measures to the Sou\v{c}ek space \cite{soucek} which is a kind of a finer completion of $W^{1,1}(\O;\R^M)$ than $BV(\O;\R^M)$. 
As an application of our results, we then present a relaxation theorem for functionals on $W^{1,1}(\O;\R^M)$ including integral terms on the boundary in Section~\ref{sec:relaxation}. Our work then closes with the Appendix which has three sections. Its first section contains a proof of an auxiliary result related to weak* lower semicontinuity along sequences of gradients which do not oscilate but concentrate at the boundary. The second section states  
a relationship between generalized Young measures and DiPerna-Majda measures often used by the second and the third author to characterize oscillations/concentrations in sequences of gradients. Finally, in the third section we explain the connection between quasiconvexity at the boundary and a variant of it which appears as an assumption in our relaxation theorem.

 \section{Notations and preliminaries}\label{sec:preliminaries} We start with some definitions and explanations of frequently used notation.

\subsection{Notation} Throughout the text, unless otherwise specified, we will use the following notation.

\begin{itemize}

\item[-] $\O\subset\R^N$, $N \geq 1,$ stands for a bounded domain with a boundary of class $C^1$.

\item[-]  $|K|$ denotes the Lebesgue measure of $K \subset \mathbb{R}^{N}$.

\item[-] $B$ denotes the open unit ball in $\R^N$ centered at zero.

\item[-] Given a unit vector $\varrho\in\R^N$, we define the unit half-ball
\begin{equation}\label{dro}D_\varrho=B\cap \{x\in\R^N;\ \varrho\cdot x<0\}
\end{equation}
\noindent and $\Gamma_\varrho :=\{x\in\partial D_\varrho;\ \varrho\cdot x=0\}\subset\partial D_\varrho$.

\item[-]  $B_r(x):=\{y:\,\, \abs{y-x}<r\}$ ($r>0$ and the center point $x$ in an Euclidean vector space, with norm $\abs{\cdot}$, that should be clear from the context).

\item[-]  $(K)_{r}:=\bigcup_{x\in K}B_r(x)$   (open $r$-neighborhood of a set $K\subset\RR^N$).

\item[-] The unit spheres centered at zero in $\R^N$ and $R^{M\times N}$ are denoted by $S^{N-1}$ and $S^{M\times N-1}$, respectively.

\item[-] Given $a\in\R^M$ and $b\in \R^N$, $a\otimes b$ is the matrix given by $(a\otimes b)_{ij}=a_ib_j$ for all $1\le i\le M$  and all $1\le j\le N$.

\item[-] Given real-valued functions $x\mapsto f(x)$ and $y\mapsto g(y)$, the function $(x,y)\mapsto (f \otimes g)(x,y)$ is defined by $(f\otimes g)(x,y):=f(x)g(y)$.

\item[-] $C$ represents a generic positive constant whose value might change from line to line.

\end{itemize}

\subsection{Space of functions} By $\rca(X)$ we denote the set of Radon measures on a Borel set $X$ and by   $\rca^+_1(X)$  its subset of  probability measures. We recall that the support of a measure $\mu$, ${\rm supp}\, \mu$, is the smallest closed set $X$
such that $\mu(A)=0$ if $X\cap A=\emptyset$.
Given a measure $\mu$, we write ``$\mu$-almost all'' or ``$\mu$-a.e.'' if  we mean ``up to a set with the $\mu$-measure zero''. If $\mu=\cL^N$, the $N$-dimensional Lebesgue measure, we usually omit writing it in the notation. As usual $\cH^{N-1}$ represents the $(N-1)$-dimensional Hausdorff measure in $\R^N$. 

Given a domain $\O\subset\R^N$
the space $C(\O)$ consists of all real-valued continuous functions defined in $\O$ and $C_0(\O)$ stands for the subset of these functions whose support is contained in $\O$.  $W^{1,p}(\O;\R^m)$, $1\le p\le +\infty$, is the usual Sobolev space of measurable mappings which together with their first (distributional) derivatives are $p^{\text th}$-integrable (if $p<+\infty$) or essentially bounded (if $p=+\infty$).   The space of functions with bounded variation, $\mathrm{BV}(\O;\R^M)$ (see \cite{AmFuPa00B}), consists of all $u \in  L^1(\Omega; \R^M)$   such that its first order distribu\-tional derivatives
$D_j u_i\in \mathcal M (\O)$. Given $u\in \mathrm{BV}(\O;\R^M)$ we denote by $Du$ the matrix-valued measure whose entries are $D_j u_i,$
  $\nabla u$ is the density of the absolutely continuous part of $Du$ with respect to the Lebesgue measure, while $Du^s$ denotes the singular part of $Du$ with polar decomposition $\md Du^s(x)=\frac{\md Du^s}{\md\abs{Du^s}}(x)\md\abs{Du^s}(x)$. We recall that $(u_n)_{n\in\NN}\subset \mathrm{BV}(\Omega;\RR^M)$ {\it converges weakly}$^*$ to $u\in\mathrm{BV}(\Omega;\RR^M)$ if $u_n\to u$ strongly in $L^1(\Omega;\RR^M)$ and $Du_n\stackrel{*}{\rightharpoonup} Du$ in $\cM(\Omega;\RR^{M\times N})$; see \cite[Def.~3.11]{AmFuPa00B}.

Given  $\mu\in\cM(\Omega)$ we denote by $|\mu|$ its total variation (norm), i.e., $|\mu|:=\sup|\int_\Omega f\,d\mu|$ where the supremum is taken over all $f\in C_0(\Omega)$ such that $\|f\|_{C_0(\Omega)}=1$.

\subsection{Quasiconvex functions and functions with linear growth} Let $\O\subset\R^N$ be a bounded Lipschitz domain. A function $v:\R^{M\times N}\to\R$ is said to be quasiconvex \cite{dacorogna} if
for any $A\in\R^{M\times N}$ and any $\varphi\in W^{1,\infty}_0(\O;\R^M)$
\be\label{quasiconvexity}
v(A)|\O|\le \int_\O v(A+\nabla \varphi(x))\,\md x\ .\ee
If $v:\R^{M\times N}\to\R$ is not quasiconvex its quasiconvex envelope 
$Qv:\R^{m\times n}\to\R$ is defined by
$$
Qv=\sup\left\{h\le v;\ \mbox{$h:\R^{M\times N}\to\R$ quasiconvex }\right\}$$
or   $Qv=-\infty$ if the set on the right-hand side is empty .
If $v$ is locally bounded and Borel measurable then for any $A\in\R^{M\times N}$ (see \cite{dacorogna})
\be\label{relaxation}
Qv(A)=\inf_{\varphi\in W^{1,\infty}_0(\O;\R^M)} \frac{1}{|\O|} \int_\O v(A+\nabla \varphi(x))\,\md x\ .\ee

Given a real-valued function $v$ on some Euclidean space with norm $\abs{\cdot}$,
 $v$ is said to have \emph{at most linear growth} if 
\[
	|v(\cdot)|\leq C(1+\abs{\cdot})\quad\text{for a constant $C\geq 0$}.
\]
Clearly ``at most linear growth''  only refers to the behavior of  $v$ for large norms of its arguments.

Throughout this work, we use nonlinear transformations of measures, more precisely, derivatives in $\mathrm{BV}$. A standard class of functions for which such nonlinear expressions can be defined is
\begin{align}\label{class}
\us:=\{v\in C(\R^{M\times N});\ 
\text{$v$ has at most linear growth} ~ \& ~ v^\infty \text{ exists }\}
\end{align} 
where
$v^\infty$ denotes the \emph{recession function} of $v$ (if it exists). We recall that a function $v$ is positively one-homogeneous
if $v(\alpha \cdot)=\alpha v(\cdot)$ for all $\alpha\ge 0$ and we observe that
 $v^\infty:\R^{M\times N}\to\R$ is the (automatically continuous and positively one-homogeneous) function defined by
\begin{align}\label{recession}
	v^\infty(A):=\lim_{\alpha\to+\infty,t\to A} \frac{1}{\alpha} v(\alpha t).
\end{align}  
If $v^\infty$ exists, it coincides with the \emph{generalized recession function} $v^\sharp$ of $v$ given by
\begin{align}\label{genrecession}
	v^\sharp(A):=\limsup_{\alpha\to+\infty} \frac{1}{\alpha}v(\alpha A).
\end{align}

Note, however, that even if $v^\sharp$ is continuous and the $\limsup$ in \eqref{genrecession} is a limit, this does not guarantee \eqref{recession}.

\subsection{Tools and notation for sequences and functionals on $\mathrm{BV}$}

Here we briefly recall some notation for nonlinear functionals on $BV$ and some results of \cite{BKK2013} which will be used below. In particular, we rely on the following notion:
\begin{defn}\label{def:charge}
Given a sequence $(u_n)\subset \mathrm{BV}(\Omega;\RR^M)$ and a closed set $K\subset \overline{\Omega}$, we say that \emph{$(D u_n)$ does not charge $K$}, if $\abs{Du_n}$ is {\it tight} in $\bar{\Omega}\setminus K$, i.e.,
$$
	\sup_{n\in\NN} \abs{D u_n}\big((K)_r\cap \Omega \big)\underset{r\to 0^+}{\To} 0.
$$

\end{defn}

\begin{lem}[Local decomposition in $\mathrm{BV}$, {\cite[Lemma 4.2]{BKK2013}}]
\label{lem:decloc}
Let $\Omega\subset \RR^N$ be open and bounded
and let $K_j\subset \overline{\Omega}$, $j=1,\ldots,J$, be a finite family of compact sets such that $\overline{\Omega}\subset \bigcup_j K_j$.
Then for every bounded sequence $(u_n)\subset \mathrm{BV}(\Omega;\RR^M)$ 
with
$u_n\to 0$ in $L^1(\Omega;\RR^M)$,
there exists a subsequence $(u_n)$ (not relabeled) which can be decomposed as
$$
	u_{n}=u_{1,n}+\ldots+u_{J,n},
$$
where for each $j\in\{1,\ldots, J\}$, $(u_{j,n})_n$ is a bounded sequence in $\mathrm{BV}(\Omega;\RR^M)$ converging to zero in $L^1$ such that the following two conditions hold for every $j$:
\begin{align*}
\begin{alignedat}[c]{2}
	&\text{(i)}~~&&
	\{u_{j,n}\neq 0\}\subset \{u_n\neq 0\},~\overline{\{u_{j,n}\neq 0\}}
	\subset (K_j)_{\frac{1}{n}}
	\setminus \textstyle{\bigcup_{i<j}}(K_i)_{\frac{1}{2n}},\\
	&&& \abs{D u_{j,n}}~\text{is absolutely continuous w.r.t.~$\abs{D u_n}+\cL^N$,}\\
    &&&\text{and $\abs{D u_{j,n}}\leq \abs{D u_n}+\tfrac{1}{n}$ as measures;}
	\\
	&\text{(ii)}~~ &&\text{$(Du_{j,n})$ does not charge 
	$\textstyle{\bigcup_{i<j}}K_i$.}
\end{alignedat}
\end{align*}
Moreover, if $\partial \Omega$ is Lipschitz and each $u_n$ has vanishing trace on $\partial \Omega$, this is inherited by  $u_{j,n}$. 
\end{lem}

As derived in \cite{BKK2013} the properties of the component sequences found in Lemma~\ref{lem:decloc} guarantee that 
at least asymptotically as $n\to\infty$, they behave as if they had pairwise disjoint support. To present this result we recall that if $v\in\us$ and $u\in \mathrm{BV}(\Omega;\RR^M)$,  $v(Du)$ is defined as the real-valued measure given by
\begin{align}
	dv(Du)(x):=v(\nabla u(x))\md x+v^\infty\Big(\frac{dDu^s}{d\abs{Du^s}}(x)\Big)\md\abs{Du^s}(x),
\end{align}

\begin{prop}[cf.~{\cite[Proposition 4.3]{BKK2013}}]\label{prop:fdecloc}
Suppose that $v\in\us$. Then for every $u\in \mathrm{BV}(\Omega;\RR^M)$ and every decomposition of a bounded sequence $u_{n}=u_{1,n}+\ldots+u_{J,n}\in \mathrm{BV}(\Omega;\RR^M)$ with the properties listed in
Lemma~\ref{lem:decloc}, we have that
\[
	\bald
	v(Du_{n}+Du)-v(Du)-\sum_{j=1}^{J} 
	\big[v\big(Du_{j,n}+Du\big)-v(Du)\big]
	\underset{n\to\infty}{\To}0&\\
	\text{in total variation of measures}&.
	\eald
\]
%In particular, 
%\be\label{pfdecloc-0}
%	F(u_{n}+v)-F(v)-\sum_{j=1}^{J} 
%	\big[F\big(u_{j,n}+v\big)-F(v)\big]
%	\underset{n\to\infty}{\To}0.
%\ee
\end{prop}
Another useful observation is that for sequences purely concentrating at the boundary (but not in the interior), it is always possible to add or remove a non-zero weak$^*$ limit:
\begin{prop}[cf.~{\cite[Proposition 5.4]{BKK2013}} and its proof]\label{prop:asympaddconc}
Suppose that $v\in\us$, and let $(c_n)_{n \in \mathbb{N}}\subset \mathrm{BV}(\Omega;\RR^M)$ be a bounded sequence such that 
$S_n:=\{c_n\neq 0\}\cup \supp \abs{D c_n} \subset (\partial \Omega)_{r_n}$ with a decreasing sequence $r_n\searrow 0$.
Then for every $c\in \mathrm{BV}(\Omega;\RR^M)$,
\[
	v(Dc+Dc_n)-v(Dc)-v(Dc_n)+v(0)\underset {n\to\infty}{\To} 0
\]
in total variation of measures.
\end{prop}

\subsection{Functions with quasi-sublinear growth from below}\label{ss:defqslb} This is a central condition for our analysis and it  first appeared in \cite{BKK2013}. 

\begin{defn}[Quasi-sublinear growth from below]\label{def:qslb}
Assume that $x_0\in\partial\O$.  We say that a function $v:\RR^{M\times N}\to \RR$ is \emph{quasi-sublinear from below} (\emph{qslb}) at $x_0$
(shortly $v\in\QSLB(x_0)$) if
\[
\bald
	&\bald[t]
	&\text{for every $\eps>0$, there exist $\delta=\delta(\eps)>0$, $C=C(\eps)\in \RR$ s.t.}\\
	&\int_{\Omega\cap B_\delta(x_0)} v(\nabla u(x))\, \d x
	~\geq~
	-\eps \int_{\Omega\cap B_\delta(x_0)} \abs{\nabla u(x)}\, \d x-C 
	\eald  \\
	&\text{for every $u\in W^{1,1}(B_\delta(x_0)\cap\Omega;\RR^M)$ with  compact support in $B_\delta(x_0)$.}
\eald
\]
\end{defn}

If $\partial\Omega$ is of class $C^1$ near $x_0\in\partial\O$ and $v\in\us$ then Definition~\ref{def:qslb} is equivalent  to the following statement:
We say that $v\in \QSLB(x_0)$  if for every $\eps>0$, there exists $C=C(\eps)\geq 0$ such that
\be\label{qslb-limit1}
\int_{D_{\varrho(x_0)}} v(\nabla u(x))\, \md x
	~\geq~
	-\eps \int_{D_{\varrho(x_0)}} \abs{\nabla u(x)}\, \md x-C 
\ee
	for all $u\in W_0^{1,1}(B;\R^M)$, where $\varrho(x_0)$ is the outer unit normal to $\partial\O$ at $x_0$.
	
	In this case, we write (with slight abuse of the notation)  $v\in\QSLB(\varrho(x_0))$ to indicate the depedence of the property of $v$  on the normal vector to the boundary.  Definition~\ref{def:qslb} corresponds to the condition used  in Theorem 1.6 (ii) in \cite{kroemer} for $p=1$
and to the notion of $1$-quasisubcritical growth from below at a boundary point as defined in \cite{KK2011}.

\begin{rem}\label{rem:qslbrecession}
As the name suggests quasi-sublinear growth from below is only affected by the behavior of $v$ for large arguments.

It holds that $v\in\QSLB(x_0)$ if and only if
$v^\infty\in\QSLB(x_0)$, assuming   $v^\infty$   (see \eqref{recession}) exists. Moreover, for
positively one-homogeneous functions such as $v^\infty$, \eqref{qslb-limit1} can be simplified significantly. More precisely, we have $v^\infty \in\QSLB(\varrho(x_0))$ if and only if
\be\label{qslb-recession}
\int_{D_{\varrho(x_0)}} v^\infty(\nabla u(x))\, \md x
	~\geq~0
\ee
for all $u\in W_0^{1,1}(B;\R^M)$. For more details, we refer to \cite[Section 3]{BKK2013}.

\end{rem}

Functions with quasi-sublinear growth from below play a key role in characterization of weak*-lower semicontinuity in $BV$. We recall the following definition.

\begin{defn}[w$^*$-lsc]\label{def:wlsc}
We say that a functional $F:\mathrm{BV}(\Omega;\RR^M) \to \RR$ is sequentially \emph{weakly$^*$-lower semicontinuous (w$^*$-lsc)} in $\mathrm{BV}(\Omega;\RR^M)$ if 
$$
\liminf_{n \to \infty} F(u_n)\geq F(u)
$$
for every sequence $(u_n)\subset \mathrm{BV}(\Omega;\RR^M)$ and such that $u_n\wstar u$ in $\mathrm{BV}$.
\end{defn}

Let $f\in C(\bar\O;\us)$ and  $F:\mathrm{BV}(\Omega;\RR^M)\to\R$ be defined as 
\begin{align}\label{BV-functional}
F(u):=\int_\O \md f (x,Du)(x)\ ,
\end{align}

 where 
\[
	\md f(x,Du):=f(x,\nabla u(x)) \d x+f^\infty\Big(x,\frac{dDu}{d\abs{Du}}(x)\Big) \d \abs{D^su}(x)\ .
\]

The proof of next proposition,
  Proposition~\ref{prop:wlsc-1}, is  implicitly contained in that of  Theorem 2.9 in \cite{BKK2013}. It shows that being of quasi-sublinear growth from below is sufficient to ensure weak* lower semicontinuity  along sequences concentrating at the boundary.  For the convenience of the reader, a  proof of this proposition can be found in the Appendix A. 

\begin{prop}
\label{prop:wlsc-1}
Assume that $f=g\otimes v$ with $(g,v)\in C(\bar\O)\times\us$, $f(x,\cdot)\in \QSLB(x)$ for all $x\in\partial\O$ and $F$ is given by \eqref{BV-functional}.  Then $F$ 
is weak*-lower semicontinuous along all sequences $(c_n)_{n \in \mathbb{N}}\subset \mathrm{BV}(\Omega;\RR^N)$ that are bounded 
and satisfy 
$S_n:=\{c_n\neq 0\}\cup \supp \abs{D c_n}\subset (\partial \Omega)_{r_n}$ for a decreasing sequence $r_n\searrow 0.$

\end{prop}

The dependence on the normal at a given boundary point is illustrated by the following lemma.
 \begin{lem}\label{lem:pqscbrotated}
Let $v:\RR^{M\times N}\to \R$ be continuous and with linear growth. Given $\rho_{1},\rho_2\in S^{N-1}$ if $R_{21}\in \R^{N\times N}$ is an orthogonal matrix such that $\nu_2=R_{21}\rho_1$ then
$v\in\QSLB(\nu_1)$  if and only if $A\mapsto v(AR_{21})\in\QSLB(\rho_2)$. 
\end{lem}

\begin{proof}
Let $\varphi_1\in W_0^{1,1}(B;\RR^M)$. %Using the notation of Subsection~\ref{ss:defqslb}, 
We have that
$$
	\int_{D_{\rho_1}} v(\nabla \varphi_1)\,\md x\geq -\eps \int_{D_{\rho_1}} \abs{\nabla \varphi_1}^p\,\md x-C_\eps
$$
if and only if for $\varphi_2\in W_0^{1,p}(B_1;\RR^M)$, $\varphi_2(y):=\varphi_1\big(R_{21}^{-1} y\big)$,
$$
	\int_{D_{\rho_2}} v\big((\nabla \varphi_2) R_{21}\big)\,dy\geq -\eps \int_{D_{\rho_2}} \abs{\nabla \varphi_2}^p\,dy-C_\eps,
$$
by the change of variables given by $y=R_{21}x$. Here, note that $D_{\rho_2}=R_{21}D_{\rho_1}$, $\abs{\det R_{21}}=1$ and 
$\abs{(\nabla \varphi_2) R_{21}}=\abs{\nabla \varphi_2}$.
\end{proof}

Similar to the well-known relationship of quasiconvexity and rank-1-convexity, the following can be said about functions in $\QSLB$:
\begin{lem}\label{lem:slb-along-rank1}
Let $\varrho\in S^{N-1}$ and let $v\in \QSLB(\varrho)$ a positively one-homogeneous function.
Then for all matrices of the form $A:=a\otimes \varrho$, with  some fixed $a\in \RR^M$, we have that
$v(A)\geq 0$. 
\end{lem}
\begin{proof}
The proof is indirect. Assume by contradiction that $v(A)<0$ for an admissible matrix $A$.
Let $\eps,r\in (0,\frac{1}{2})$ and define 
\begin{align*}
	&S(r):=\{ y\in \partial D_{\varrho},~\abs{y}<r\}\subset \partial D_{\varrho}
\end{align*}
(see \eqref{dro}) and let
\begin{align*}
	&Z^-(r,\eps):=\{ y-t \varrho \mid y\in \partial D_{\varrho},~\abs{y}<r,~0<t<\eps\}\subset D_{\varrho},\\
	&Z^+(r,\eps):=\{ y+t \varrho \mid y\in \partial D_{\varrho},~\abs{y}<r,~0<t<\eps\}\subset D_{\varrho},
\end{align*}
denote the two cylinders of height $\eps$ on either side of their circular base $S(r)$ located in the flat part of $\partial D_{\varrho}$. Finally, let
$Z(r,\eps)$ denote the interior of $\overline{Z^-(r,\eps)\cup Z^+(r,\eps)}$.
For fixed $r>0$, we choose $\varphi_\eps\in W_0^{1,\infty}(B)$ in such a way that
\[
	\nabla \varphi_\eps=A~~\text{in}~Z^-(r,\eps),
	~~\nabla \varphi_\eps=-A~~\text{in}~Z^+(r,\eps),
	~~\nabla \varphi_\eps=0~~\text{in}~B\setminus \overline{Z(r+\eps,\eps)},
\]
and
\[	
	\abs{\nabla \varphi_\eps}\leq \abs{A}~~\text{in}~
	Z(r+\eps,\eps)\setminus \overline{Z(r,\eps)},
\]
which is, basically, (for $r>>\eps$) a long, narrow ``tent'' on $Z(r,\eps)$ closed with ``face walls'' of similar slope above $Z(r+\eps,\eps)\setminus \overline{Z(r,\eps)}$. Because of $v^\infty\in \QSLB(\varrho)$, we have
\begin{align*}
	0 &\leq \frac{1}{\eps} \int_{D_\varrho} v(\nabla \varphi_\eps)\,\md x\\
	&=	\frac{1}{\eps} \abs{Z^-(r,\eps)}v(A)+
	\frac{1}{\eps}\int_{Z^-(r+\eps,\eps)\setminus \overline{Z^-(r,\eps)}} v(\nabla\varphi_\eps)\,\md x\\
	&\leq \cH^{N-1}(S(r))v(A) +
	\cH^{N-1}(S(r+\eps)\setminus S(r))\abs{A} \norm{v}_{L^\infty(S^{M\times N-1})}.
\end{align*}
Since $v(A)<0$ and $\cH^{N-1}(S(r+\eps)\setminus S(r))\leq C\eps$, we obtain a contradiction in the limit as $\eps\to 0$. 
\end{proof}
\begin{remark}
For general functions $v\in \QSLB(\varrho)$ possessing a recession function $v^\infty$, 
we automatically have $v^\infty\in \QSLB(\varrho)$, and by continuity, $1$-homogeneity and \eqref{recession},
\[
	v^\infty(A)=\lim_{t\to+\infty} \frac{1}{t}v^\infty(\xi+tA)=\lim_{t\to+\infty} \frac{1}{t}v(\xi+tA)
\]
for all $\xi,A\in \RR^{M\times N}$.
For $v\in \QSLB(\varrho)$, the assertion of Lemma~\ref{lem:slb-along-rank1} 
therefore generalizes to
\[
	\limsup_{t\to+\infty} \frac{1}{t}v(\xi+t(a\otimes \varrho))\geq 0,
	\quad\text{for all $\xi\in\RR^{M\times N}$, $a\in \RR^M$},
\]
i.e., $v$ has sublinear growth from below along rank-1 lines corresponding to the normal $\varrho$.
\end{remark}

 \subsection{Young measures and their generalizations}

\-

\subsubsection{Young measures: the sets $\mathcal{Y}(\O;\R^{M\times N})$ and $Y(\O;\R^{M\times N})$}

\-

{\it Young
measures} on a bounded  domain $\O\subset\R^N$ are weakly*-measurable mappings
$x\mapsto\nu_x$, $x\in\O$, with $\nu_x:\O\to\rca^+_1(\R^{M\times N})$. We recall that the adjective ``weakly*-measurable'' means that,
for any $v\in C_0(\R^{M\times N})$, the mapping
$x\mapsto\int_{\R^{M\times N}} v(A)\md\nu_x(A)$, $x\in\O$, is
measurable in the usual sense.  Let us also remind here that, by the Riesz theorem,
$\rca(\R^{M\times N})$, normed by the total variation, is a Banach space which is
isometrically isomorphic with $C_0(\R^{M\times N})^*$ (the dual space of $C_0(\R^{M\times N})$).

We denote the set of all Young measures by $\mathcal{Y}(\O;\R^{M\times N})$. It
is known that $\mathcal{Y}(\O;\R^{M\times N})$ is a convex subset of $L^\infty_{\rm
w*}(\O;\rca(\R^{M\times N}))\cong L^1(\O;C_0(\R^{M\times N}))^*$, where the subscript ``w*''
indicates the property ``weakly*-measurable''. 

Let
\begin{equation}\label{C1class}
C_1(\R^{M\times N})=\{v\in C(\R^{M\times N}); \lim_{|A|\to +\infty} v(A)/|A|=0\}\ .
\end{equation}
 \noindent As it was shown in 
\cite{ball3, tartar1,valadier}, for every bounded sequence $(Y_k)_{k\in\N}\subset L^1(\O;\R^{M\times N})$, there exists a subsequence (not relabeled) and a Young measure
$\nu=\{\nu_x\}_{x\in\O}\in\mathcal{Y}(\O;\R^{M\times N})$ such that for all $v\in C_1(\R^{M\times N})$ (see \eqref{C1class}) and all $g\in L^\infty(\O)$ 
\begin{align}\label{one}
\lim_{k\to\infty}\int_\O g(x)v(Y_k(x))\,\md x=\int_\O g(x)\left[\int_{\R^{M\times N}}v(A)\md\nu_x(A)\right]\,\md x\ .
\end{align}

Let us denote by $Y(\O;\R^{M\times N})$ the
set of all Young measures which are created in this way, i.e., by considering
all bounded sequences in $L^1(\O;\R^{M\times N})$. 

If $v$ has linear growth at infinity the limit passage \eqref{one} generally does not hold due to concentration effects created by $(Y_k)$ if this sequence is not uniformly integrable. For this reason, various generalizations have been invented to allow for a description of limits in \eqref{one}. We refer to \cite{r} for a survey of these approaches.
The authors in \cite{ab} showed that if $(Y_k)_{k\in\N}\subset L^1(\O;\R^{M\times N})$ is a bounded sequence then there is a (non-relabeled) subsequence, a Young measure $\nu\in Y(\O;\R^{M\times N})$,  a positive measure $\lambda\in\mathcal{M}(\bar\O)$ and  a mapping
 $\nu^\infty\in  L^\infty_{\rm w*}(\bar\O,\lambda;\rca(S^{M\times N-1}))$ with values in probability measures and defined for $\lambda$-almost all $x\in\bar\O$ such that  for all $g\in C(\bar\O)$ and all $v\in \us$ (c.f.~\eqref{class})
\begin{align}\label{two}
\lim_{k\to\infty}\int_\O g(x)v(Y_k(x))\,\md x&=\int_\O g(x)\left[\int_{\R^{M\times N}}v(A)\md\nu_x(A)\right]\,\md x\nonumber\\
&+\int_{\bar\O}g(x)\left[\int_{S^{M\times N-1}}v^\infty(A)\md\nu^\infty_x(A)\right]\md\lambda(x)\ .
\end{align}
Roughly speaking, $\lambda$ measures how much $(|Y_k|)$ concentrates and $\nu^\infty$ accounts for the
 spatial distribution of concentrations. A closely related  description  can also be found in \cite{fonseca} and \cite{fmp}.

Continuous functions possessing recession function can be continuously extended by radial limits to the compactification of $\R^{M\times N}$ by the sphere. Indeed, 
defining  $d:\R^{M\times N}\to B$, $d(A):=A/(1+|A|)$ we can identify the compactification by the sphere with $\bar B$. Conversely, to any continuous function on $\bar B$ we can assign a function on $\R^{M\times N}$ which has the recession function.  
 Notice also that positively  one-homogeneous functions are fully determined by their values on $S^{M\times N-1}$. If $v\in\us$ then we can assume that 
$$
\frac{v(A)}{1+|A|}=
\begin{cases}  c+ v_{0,0}(A)+v_{0,1}\left(\frac{A}{|A|}\right)\frac{|A|}{1+|A|} & \mbox{ if $A\ne 0$,}\\
 c+ v_{0,0}(0) \mbox { if $A=0$}\ ,
\end{cases}
$$
where $v_{0,0}\in C_0(\R^{M\times N})$, $v_{0,1}\in C(S^{M\times N-1})$ and $c\in\R$. 
Consequently, for $g\in C(\bar\O)$ the map $(x,s)\mapsto g(x)v(s)/(1+|s|)\in C(\bar\O)\otimes C(\bar B)$. By the Stone-Weierstrass Theorem equality \eqref{two} also holds for test functions of the form
$(x,A)\mapsto f(x,A)=f_0(x,A/(1+|A|))(1+|A|)$ where $f_0\in C(\bar\O\times\bar B)$. 
Notice that such a function $f$ always has recession function $f^\infty(x,A)=\abs{A} f_0(x,\frac{A}{|A|})$, 
satisfying an even stronger version of \eqref{recession} additionally involving a kind of uniform convergence in $x$, namely,
\begin{align}\label{recessionx}
	f^\infty(x,A)=\lim_{\alpha\to+\infty,t\to A,y\to x} \frac{1}{\alpha} f(y,\alpha t).
\end{align} 
In particular, $f^\infty$ is always continuous in $x$. In general, \eqref{two} cannot be expected to hold with test functions $f$ having a (generalized) recession function with discontinuities in $x$.

\subsubsection{The set of generalized Young measures $\genYM(\bar\O;\R^{M\times N})$.} We denote by $\Lambda:=(\nu,\lambda,\nu^\infty)$  the triple appearing in \eqref{two}
and the set of all such triples is represented by $\genYM(\bar\O;\R^{M\times N})$, or briefly $\genYM$. 
For the right hand side of \eqref{two}, given $g\in C(\bar\O)$ and all $v\in \us$, we occasionally use the short-hand notation
\begin{align}\label{GYM-passage}
 \langle \langle \Lambda, g\otimes v \rangle \rangle:= \int_{\Omega} g(x) \langle \nu_x,v \rangle \,\md x +\int_{\bar \Omega} g(x)
 \langle \nu_x^{\infty}, v^{\infty} \rangle \,\md{\lambda}(x),
\end{align}
where
$$
	\langle \nu_x,v \rangle:=\int_{\R^{M\times N}}v(A)\md\nu_x(A)
	\quad\text{and}\quad
	\langle \nu_x^{\infty}, v^{\infty} \rangle:=\int_{S^{M\times N-1}}v^\infty(A)\md\nu^\infty_x(A).
$$
Also note that actually, $\nu_x$ is only defined for a.e.~$x\in \Omega$ (or, equivalently, a.e.~$x\in \bar\O$, since $\partial \O$ has vanishing Lebesgue measure), and $\nu_x^\infty$ is only defined for $\lambda$-a.e.~$x\in\bar\O$. 
Accordingly, given $\Lambda_1,\Lambda_2\in \genYM$, $\Lambda_1=(\nu_1,\lambda_1,\nu_1^\infty)$, $\Lambda_2=(\nu_2,\lambda_2,\nu_2^\infty)$, with a slight abuse of notation we write  
$\Lambda_1=\Lambda_2$ if $\nu_{1,x}=\nu_{2,x}$ for a.e.~$x\in\O$, $\lambda_1=\lambda_2$ and 
$\nu^\infty_{1,x}=\nu^\infty_{2,x}$ for $\lambda_1$ or $\lambda_2$-a.e.~$x\in\bar\O$.

\begin{definition}[Generation]
We say that \emph{$\Lambda=(\nu,\lambda,\nu^\infty)$} is generated by a bounded sequence $(Y_k)\subset L^1(\O;\R^{M\times N})$ whenever \eqref{two} holds. 
\end{definition}

We observe that this notion also makes sense for slightly more general sequences, namely, if $(Y_k)$ is a sequence of $\RR^{M\times N}$-valued Radon measures on $\O$ with bounded total variation. In that case, we say that \emph{$\Lambda=(\nu,\lambda,\nu^\infty)$ is generated by $(Y_k)$} if
\begin{align}\label{twoM}
\lim_{k\to\infty}\int_{\bar\O} g(x)\md v(Y_k)(x)=\langle \langle \Lambda,g\otimes v \rangle \rangle 
\end{align}
for every $g\in C(\bar\O)$ and every $v\in\us$. Analogous to the special case of $BV$ derivatives mentioned before here
\[
	\md v(Y_k)(x):=v\Big(\frac{d Y_k}{d \cL^N}(x)\Big)\,\md x+v^\infty\Big(\frac{dY_k}{d\abs{Y_k}}(x)\Big)\,\md \abs{Y_k^s}(x)
\]
where $\frac{d Y_k}{d \cL^N}$ denotes the density of the absolutely continuous part of $Y_k$ w.r.t.~the Lebesgue measure, $Y_k^s$ is its  singular part and $\frac{dY_k}{d\abs{Y_k}}$ stands for the density of $Y_k$ with respect to its total variation $\abs{Y_k}$.

\begin{ex}\label{ex:pointconc}
Suppose that $0\in \bar{\Omega}$, take $Y\in L^1(B;\R^{M\times N})$ and extend it by zero to the whole space $\Omega$ (without renaming it). Define $Y_k(x):=k^N Y(kx)$ for all $k\in\N$. 
Clearly $(Y_k)_{k\in\N}$ is uniformly bounded in $L^1(\Omega;\R^{M\times N})$ and $Y_k\to 0$ in measure. Consequently, $(Y_k)$ generates $\Lambda\in \genYM$ such that 
$\nu_x=\delta_0$, $\lambda=\|Y\|_{L^1}\delta_0$ and 
$$\int_{S^{M\times N-1}} v^{\infty}(A)\md\nu^\infty_0(A)=\|Y\|_{L^1}^{-1}\int_{D_{\rho}} v^{\infty}(Y(z))\,\md z\ .$$
Notice that in this particular example, $x=0$ is the only relevant point for $\nu^\infty_x$ in $\bar \Omega$ because $\lambda(\bar \Omega\setminus \{0\})=0$.
This kind of example is also possible for gradients, i.e., if $Y=\nabla u$ for some $u\in W^{1,1}(\Omega;\RR^M)$. In this case, we let $Y_k=\nabla u_k$ with $u_k(x):=k^{N-1} u(kx)$, and consequently, $\Lambda$ is a generalized gradient Young measure as defined below.
\end{ex}

\subsubsection{Piecewise generation of generalized Young measures}

The natural embedding of a function space containing a generating sequence, say, $L^1(\Omega;\R^{M\times N})$, into
 $\genYM(\bar\O;\R^{M\times N})$, is not linear, because in $\genYM$, addition and multiplication with scalars is carried out \emph{after} the application of a nonlinear test function $v$. Nevertheless, in some special cases, a kind of additivity can be observed. Essentially, we need two generating sequences lacking any kind of interaction. In terms of the generated generalized Young measure, this is made precise by requiring them to be orthogonal in the following measure theoretic sense:
\begin{defn}\label{def:genYMorthogonal}
Given $\Psi=(\eta_x,\mu,\eta_x^\infty),\Theta=(\vartheta_x,\sigma,\vartheta_x^\infty) \in \genYM$, we write that $\Psi \perp \Theta$ if there are two disjoint Borel sets $S,T\subset \bar{\O}$ such that $\bar{\O}=S\cup T$,
$\mu(T)=0=\cL^N(\{x\in T\mid \eta_x\neq \delta_0\})$ and
$\sigma(S)=0=\cL^N(\{x\in S\mid \vartheta_x\neq \delta_0\})$.
\end{defn}
We now claim that for such a pair, the sum of the respective generating sequences will always generate 
the sum $\chi_S\Psi+\chi_T\Theta$. The key ingredient for the proof is the following uniform continuity property:
\begin{prop}[special case of {\cite[Proposition 4.5]{BKK2013}}]\label{prop:measfunc-ucont}
Let $v\in \us$ and $g\in C(\bar\O)$.
Then for every pair of sequences $(\mu_k),(\lambda_k)\subset \cM(\O;\RR^{M\times N})$ such that
$\abs{\mu_k}(\O)$ and $\abs{\lambda_k}(\O)$ are bounded,
$$
	\abs{\mu_k-\lambda_k}\to 0~~\text{implies that}~~\int_{\O}g(x)dv(\mu_k)(x)-\int_{\O}g(x)dv(\lambda_k)(x)\to 0.
$$
\end{prop}
\begin{lem}[additive behavior along generating sequences for orthogonal generalized Young measures]\label{lem:gYMadd}
Let $\Psi=(\eta_x,\mu,\eta_x^\infty),\Theta=(\vartheta_x,\sigma,\vartheta_x^\infty) \in \genYM$ be such that $\Psi \perp \Theta$,
and let $\bar{\O}=S\dot\cup T$ denote the associated decomposition of $\O$ of Definition~\ref{def:genYMorthogonal}.
Moreover, let $(w_k),(z_k)\subset L^1$ be two bounded sequences generating $\Psi$ and $\Theta$, respectively.
Then $(w_k+z_k)$ generates 
\[
  \chi_S \Psi +\chi_T \Theta:=  
  (\chi_S(x)\eta_x+\chi_T(x)\vartheta_x,
  \mu(S\cap \cdot)+\sigma(T\cap \cdot),
  \chi_S(x)\eta_x^\infty+\chi_T(x)\vartheta_x^\infty).
\]
\end{lem}
\begin{remark}\label{rem:gYMadd}
Lemma~\ref{lem:gYMadd} also holds for more general generating sequences in $\cM$ instead of $L^1$. 
The proof remains the same, apart from the fact that expressions like, say, $v(w_k)\md x$ have to be replaced by $dv(w_k)(x)$.
\end{remark}
\begin{proof}[Proof of Lemma~\ref{lem:gYMadd}]
Let $\Lambda=(\nu_x,\lambda,\nu_x^\infty)$ be the generalized Young measure generated by $(w_k+z_k)$ (or a suitable subsequence, not relabeled).
Since $(\Psi,S,(w_k))$ and $(\Theta,T,(z_k))$ are interchangeable, it suffices to show that $\Lambda=\Psi$ on $S$, i.e., $\nu_x=\eta_x$ for $\cL^N$-a.e.~$x\in S$, $\lambda|_S=\mu|_S$ and $\nu_x^\infty=\eta_x^\infty$ for $\lambda$-a.e.~$x\in S$. 

By the inner regularity of the Borel measures $\tilde{\mu},\tilde{\sigma}$ defined by
\[
  \md\tilde{\mu}(x):=\md\mu(x)+\la \eta_x,\abs{\cdot}\ra\,\md x,\quad
  \md\tilde{\sigma}(x):=\md\sigma(x)+\la \vartheta_x,\abs{\cdot}\ra\,\md x,
\]
there exist increasing sequences of compact sets $S^{(m)}\subset S$ and $T^{(m)}\subset T$ such that for all $m\in \NN$,
\[
\begin{aligned}
  \abs{\tilde{\mu}(S^{(m)})-\tilde{\mu}(S)}\leq \frac{1}{m} \quad\text{and}\quad
  \abs{\tilde{\sigma}(T^{(m)})-\tilde{\sigma}(T)}\leq \frac{1}{m}.
\end{aligned}
\]
Since $\tilde{\mu}(S\setminus \bigcup_{m\in\NN} S^{(m)})=0$, it now suffices to show that $\Lambda=\Psi$ on $S^{(m)}$ for all $m$.

Suppose by contradiction that $\Lambda\neq \Psi$ on $S^{(m)}$ for an $m\in \NN$. Consequently, there exist 
$\eps>0$, $g\in C(S^{(m)})$ and $v\in \us$ such that
\begin{align}\label{pgYMadd-c0}
  \eps<\abs{\int_{S^{(m)}} g(x) \big(\la \eta_x,v\ra \md x+\la \eta^\infty_x,v^\infty\ra \md\mu(x)\big)
  -\int_{S^{(m)}} g(x) \big(\la \nu_x,v\ra \md x+\la \nu^\infty_x,v^\infty\ra \md\lambda(x)\big)}
\end{align}
By definition of $\us$, there is a constant $C>0$ such that 
$\abs{v(\cdot)}\leq C(1+\abs{\cdot})$ and $\absn{v^\infty(\cdot)}\leq C\abs{\cdot}$.
Due to dominated convergence, there exists a neighborhood $V$ of $S^{(m)}$ in $\bar\O$ such that
\begin{align}\label{pgYMadd-c1}
\begin{aligned}
   \norm{g}_{L^\infty(S^{(m)})} 
  \int_{V\setminus S^{(m)}} C\Big((1+\la \overline{\nu}_x,\abs{\cdot}\ra)\,\md x
   +\la \overline{\nu}^\infty_x,\abs{\cdot}\ra \md\overline{\lambda}(x)
  \Big) & \leq \frac{\eps}{3},
\end{aligned}
\end{align}
where $\overline{\Lambda}:=(\overline{\nu}_x,\overline{\lambda},\overline{\nu}^\infty_x)$ denotes the generalized Young measure generated by $(\abs{w_k}+\abs{z_k})$ (up to a subsequence). Observe that besides for $\overline{\Lambda}$, \eqref{pgYMadd-c1} also holds for any generalized Young measure generated by a sequence $(y_k)$ such that $\abs{y_k}\leq \abs{w_k}+\abs{z_k}$ a.e.~in $\O$.
Without changing notation, we extend $g$ to a function in $C(\bar\O)$ in such a way that $g=0$ on $\bar\O\setminus V$ and $\norm{g}_{L^\infty(\bar\O)}\leq \norm{g}_{L^\infty(S^{(m)})}$.

On the other hand,
for any function $\varphi\in C(\bar\O)$, we have
\[
  \int_{\O} \varphi(x)\abs{z_k}\,\md x\to 
  \la\la \Theta, \varphi\otimes \abs{\cdot} \ra\ra=\int_{\bar\O} \varphi(x)\,\md \tilde{\sigma}(x).
\]
%In particular, $\norm{w_k}_{L^1(\O)}\to \tilde{\mu}(\bar\O)=\tilde{\mu}(S)$. In addition,
In particular, for arbitrary but fixed $m,h\in \NN$ and any set $U=U(m,h)\subset \bar\O$ which is relatively open with respect to $\bar\O$ and satisfies $S^{(m)}\subset U \subset \overline{U} \subset \bar\O\setminus T^{(h)}$, we may choose $0\leq \varphi \leq 1$ such that $\varphi=1$ on $U$ and $\varphi=0$ on $T^{(h)}$, which yields that
\begin{align}\label{pgYMadd-c2a}
	\limsup_k \int_{U} \abs{z_k}\,\md x \leq \tilde{\sigma}(\bar\O\setminus T^{(h)}) 
	=\tilde{\sigma}(T\setminus T^{(h)}) \leq \frac{1}{h}.
\end{align}
By \eqref{pgYMadd-c2a} and Proposition~\ref{prop:measfunc-ucont}, we can choose $h\in \NN$ large enough, together with an associated neighborhood $U=U(m,h)$ of $S^{(m)}$ in $\bar\O$, w.l.o.g.~$U\subset V$, such that 
\begin{align}\label{pgYMadd-c2}
  \limsup_k
  \abs{\int_{\O} g(x)v(w_k+\chi_U z_k)\,\md x
  -\int_{\O} g(x)v(w_k)\,\md x}\leq \frac{\eps}{3}
\end{align}
Recall that $(w_k)$ generates $\Psi$, and let $\Lambda_U$ denote the generalized Young measure generated by $(w_k+\chi_U z_k)$ (up to a subsequence). Passing to the limit inside the modulus in \eqref{pgYMadd-c2}, we get  that
\begin{align}\label{pgYMadd-c3}
  \abs{\la\la \Lambda_U, g \otimes v \ra\ra
  -\la\la \Psi, g \otimes v \ra\ra}\leq \frac{\eps}{3}
\end{align}
As remarked above just below \eqref{pgYMadd-c1}, besides for $\overline{\Lambda}$, \eqref{pgYMadd-c1} also holds for the generalized Young measures $\Lambda_U=(\nu_{U,x},\lambda_U,\nu_{U,x}^\infty)$ and $\Psi$. 
Since $g=0$ outside of $V$, we can apply these two estimates in \eqref{pgYMadd-c3}, and infer that
\begin{align}\label{pgYMadd-c4}
  \abs{
  \int_{S^{(m)}} g(x) \big(\la \nu_{U,x}, v\ra\,\md x+\la \nu_{U,x}^\infty, v\ra\,\md\lambda_U(x)\big)
  -\int_{S^{(m)}} g(x) \big(\la \eta_x, v\ra\,\md x+\la \eta^\infty_x, v\ra\,\md \mu(x)\big)
  }\leq \eps.
\end{align}
Since $S^{(m)}$ is compactly contained in $U$, the generating sequences of $\Lambda$ and $\Lambda_U$ coincide on a whole neighborhood of $S^{(m)}$, whence $\Lambda=\Lambda_U$ on $S^{(m)}$. Therefore, \eqref{pgYMadd-c4} contradicts
\eqref{pgYMadd-c0}.
\end{proof}

\subsubsection{Generalized gradient Young measures: $\GYM(\bar\O;\R^{M})$}

\-

The main topic of this article is an explicit characterization of  the subclass of generalized Young measures generated by gradients. 

\begin{defn}[generalized gradient Young measures]\label{def:GYM}
We call $\Lambda:=(\nu,\lambda,\nu^\infty)\in \genYM$ a \emph{generalized $BV$-gradient Young measure} if it is generated by a sequence of derivatives in $BV$, i.e., if \eqref{twoM} holds with $Y_k=D u_k$ for a bounded sequence $(u_k)\subset BV(\Omega;\RR^M)$. The class of all generalized $BV$-gradient Young measures is denoted by $\GYM(\bar\O;\R^{M})$, or briefly $\GYM$.
\end{defn}

\begin{rem}\label{rem:generator}
As already pointed out in \cite[Proposition 4]{KriRin_YM_10}, every $\Lambda\in \GYM$ can be generated by $(\nabla u_k)$, for a suitable bounded sequence $(u_k)\in W^{1,1}(\Omega;\RR^M)$. In other words, $\GYM$ is also  the class of all generalized $W^{1,1}$-gradient Young measures.
\end{rem}

For generalized gradient Young measures satisfying $\lambda(\partial\O)=0$, an explicit characterization is provided by the following result.

\begin{thm}[{\cite[Theorem 9]{KriRin_YM_10}}]\label{thm:KriRi}
Let $\O\subset \RR^N$ be a bounded Lipschitz domain, and let $\Lambda:=(\nu,\lambda,\nu^\infty)\in \genYM(\bar\O;\RR^{M\times N})$ such that $\lambda(\partial\O)=0$. Then $\Lambda\in \GYM$ if and only if the following conditions hold:
\begin{align*}
\text{(i)}~&~ \int_{\O} \langle \nu_x,\abs{\cdot} \rangle \,\md x +\int_{\bar \Omega} 
 \langle \nu_x^{\infty}, \abs{\cdot} \rangle \,\md{\lambda}(x)<\infty,\\
\intertext{and there exists $u\in BV(\Omega;\RR^M)$ and $\omega_i\subset\O$, $\mathcal{L}^N(\omega_i)=0$, such that for all $v:\RR^{M\times N}\to \RR$ continuous and  quasiconvex with 
at most linear growth,}
\text{(ii)}~&~v(\nabla u(x)) \leq  \langle \nu_x,v \rangle+\langle \nu_x^\infty,v^\sharp \rangle \frac{\md\lambda}{\md\cL^N}(x)~~\text{for all}~x\in \Omega\setminus\omega_i,\\
\text{(iii)}~&~v^\sharp\Big(\frac{\md Du^s}{\md\abs{Du^s}}(x)\Big)\abs{Du^s}\leq \langle \nu_x^\infty,v^\sharp\rangle \lambda^s\vert_\O\quad\text{as measures.}
\end{align*}
Here $\lambda^s$ denote the singular part of  $\lambda$ with respect to the Lebesgue measure $\cL^N$, $\frac{\md\lambda}{d\cL^N}$ is the density of the absolutely continuous part of $\lambda$ w.r.t.~$\cL^N$,
$v^\infty$ and $v^\sharp$ are the recession function of $v$ and the generalized recession function of $v$ as introduced in \eqref{recession} and \eqref{genrecession}, respectively.
\end{thm}
\begin{remark}\label{rem:KR-Du}
As already pointed out in \cite{KriRin_YM_10}, (ii) and (iii) imply that $Du$ is the center of mass of $\Lambda$,
i.e., $dDu(x)= \langle \nu_x, id \rangle\,\md x+\langle \nu_x^\infty, id \rangle\,\md \lambda(x)$.
\end{remark}
 
%%%%%%%%%%%%%%%%%%%%%%%%%%%%%%%%%%%%%%%%%%%%%%%%%%%%%%%%%%%%%%%%%%%%%%%%%%%%%%%%%%%%%%%%%%%%%%

\section{The main result}\label{sec:main}
We prove the following theorem which fully characterizes elements of $\GYM$.
\begin{thm}[characterization of $\GYM$ up to the boundary]\label{thm:main}
Let $\O\subset \RR^N$ be a bounded domain of class $C^1$, and let $\Lambda:=(\nu,\lambda,\nu^\infty)\in \genYM(\Omega;\RR^{M\times N})$. Then $\Lambda\in \GYM$ if and only if 
the following conditions hold:
\begin{align*}
\text{(i)}~&~ \int_{\O} \langle \nu_x,\abs{\cdot} \rangle \,\md x +\int_{\bar \Omega} 
 \langle \nu_x^{\infty}, \abs{\cdot} \rangle \,\md{\lambda}(x)<\infty,\\
\intertext{there exists $u\in BV(\Omega;\RR^M)$ and $\omega_i\subset\O$, $\mathcal{L}^N(\omega_i)=0$, such that for all $v:\RR^{M\times N}\to \RR$ continuous and  quasiconvex with 
at most linear growth,}
\text{(ii)}~&~v(\nabla u(x)) \leq  \langle \nu_x,v \rangle+\langle \nu_x^\infty,v^\sharp \rangle \frac{\md\lambda}{\md\cL^N}(x)~~\text{for all}~~x\in \Omega\setminus\omega_i,\\
\text{(iii)}~&~v^\sharp\Big(\frac{\md Du^s}{\md\abs{Du^s}}(x)\Big)\abs{Du^s}\leq \langle \nu_x^\infty,v^\sharp\rangle \lambda^s\vert_\O\quad\text{as measures on $\O$,}\\
\intertext{and there is $\omega_b\subset\partial\O$, $\lambda(\omega_b)=0$ such that }
\text{(iv)}~&~ 0\leq \langle \nu_x^\infty,v^\infty\rangle \quad\text{for all $x\in\partial \O\setminus\omega_b$ and all $v \in \us\cap \QSLB(\varrho(x))$,}
\end{align*}
 where $\varrho(x)$ is the outer unit normal to $\partial\O$ at $x$.
\end{thm}
\begin{remark}\label{rem:thmmainN1}
If $N=1$, it is clear that $\GYM=\genYM$, simply because all functions on a 1d-domain are gradients of a potential. In Theorem~\ref{thm:main}, (i)---(iv) always hold for $N=1$. More precisely, (i)--(iii) then are a consequence of Jensen's inequality, since quasiconvexity reduces to convexity (also recall that $Du$ is the center of mass of $\Lambda$, cf.~Remark~\ref{rem:KR-Du}). Similarly, (iv) becomes trivial because $v\in \QSLB$ is equivalent to $v^\infty\geq 0$ for $N=1$.
\end{remark}
By Remark~\ref{rem:thmmainN1}, it suffices to prove Theorem~\ref{thm:main} for $N\geq 2$.
The proof starts in Section~\ref{sec:separate} below, where we split any given $\Lambda\in \genYM$ into two essentially disjoint pieces: an inner part $\Lambda_i$ and a boundary part $\Lambda_b$. Due to Proposition~\ref{prop:split}, it will be enough to find conditions characterizing 
$\Lambda_i\in \GYM$ and $\Lambda_b\in \GYM$, separately (see Proposition \ref{nec-lb} and Proposition \ref{prop:suff-boundary}).
By Theorem~\ref{thm:KriRi} and the definition of $\Lambda_i$, conditions (i)--(iii) characterize $\Lambda_i$ as an element of $\GYM$. The proof of Theorem~\ref{thm:main} will be completed by showing that (iv) is a necessary and sufficient condition for $\Lambda_b\in \GYM$  (see Section~\ref{sec:nec} and Section~\ref{sec:suff}).

\section{Separating boundary and interior}\label{sec:separate}
The aim of this section is to show that when studying generalized gradient Young measures, 
it is enough to look at interior and boundary parts separately. For this purpose,
we decompose $\Lambda:=(\nu,\lambda,\nu^\infty)\in \genYM$
into two parts: the inner part $\Lambda_i=(\nu_i,\lambda_i,\nu_i^\infty)$ and the boundary part $\Lambda_b=(\nu_b,\lambda_b,\nu_b^\infty)$ defined  as follows:
\begin{align}\label{Lamdai}
	\langle \langle \Lambda_i, g\otimes v \rangle \rangle:= \int_{\Omega} g(x) \langle \nu_x,v \rangle \,\md x +\int_{ \Omega} g(x)
 \langle \nu_x^{\infty}, v^{\infty}  \rangle \,\md{\lambda}(x)\quad\text{for all $g\in C(\bar\O)$, $v\in\us$,}
\end{align}
i.e., $\Lambda\vert_\Omega=\Lambda_i\vert_\Omega$ and $\lambda_i(\partial\O)=0$, and
\begin{align}\label{Lamdab}
	\langle \langle \Lambda_b, g\otimes v \rangle \rangle:= \int_{\Omega} g(x)v(0) \,\md x +\int_{\partial\Omega}
 g(x)\langle \nu_x^{\infty}, v^{\infty} \rangle \,\md{\lambda}(x)\quad\text{for all $g\in C(\bar\O)$, $v\in\us$,}
\end{align}
i.e., $\Lambda\vert_{\partial\O}=\Lambda_b\vert_{\partial\O}$ while $\Lambda_b\vert_{\O}$ is the ``trivial'' measure $(\delta_0,0,-)$ generated by the constant sequence $0\in\RR^{M\times N}$ (here, recall that $\nu_x^\infty$ can be ignored on sets with $\lambda$-measure zero).
Notice that (cf. \eqref{GYM-passage})
\begin{align}\label{Lamdasplit}
\langle\langle\Lambda,g(x)v\rangle\rangle=\langle\langle\Lambda_i,g(x)v\rangle\rangle+\langle\langle\Lambda_b,g(x)v\rangle\rangle-\int_{\Omega} g(x)v(0) \,\md x\ .
\end{align}
\begin{remark}
In a sense, the term $\int_{\Omega} g(x)v(0)\,\md x$ in the definition \eqref{Lamdab} of $\Lambda_b=(\nu_b,\lambda_b,\nu_b^\infty)$ is artificial. Without it, however, $\nu_{b,x}$ would be the zero measure, and to have $\Lambda_b\in\genYM$, we need in particular $\nu_{b,x}$ to be a probability measure for a.e.~$x\in \O$. We therefore
choose $\nu_{b,x}=\delta_0$ for $x\in \O$, which corresponds to a generating sequence purely concentrating at the boundary.
\end{remark}

\begin{prop}\label{prop:split}
Let $\Lambda:=(\nu,\lambda,\nu^\infty)\in \genYM$. Then $\Lambda\in \GYM$ if and only if both $\Lambda_i\in \GYM$ and $\Lambda_b \in \GYM$. 
%Moreover, if $\Lambda$ is generated by $(Du_n)$ for a bounded sequence $u_n\rightharpoonup^*u$ in $BV(\Imega;\RR^M)$, then
%$\Lambda_i$ and $\Lambda_b$ are generated by $(Dd_n+Du)$ and $(Dc_n)$, where $c_n$ and $d_n$ are given by the decomposition $u_n-u=:\tilde{u_n}=c_n+d_n$ given by Lemma~\ref{lem:decloc} applied to 
%$\tilde{u_n}$ with $J=2$ and the sets $K_1:=\partial \O$, $K_2:=\bar\O$.
\end{prop}
\begin{proof}
``only if'': Assuming that $\Lambda\in \GYM$ let us see that $\Lambda_i\in \GYM$ and $\Lambda_b \in \GYM$. For this purpose let $(u_n)\subset BV(\Omega;\RR^M)$ be a bounded sequence such that $(D u_n)$ generates $\Lambda$. 
Without loss of generality, we may assume that $u_n\rightharpoonup^* u$ with some $u\in BV$.
Using Lemma~\ref{lem:decloc} with $J=2$, $K_1:=\partial\Omega$ and $K_2:=\bar\Omega$, we decompose
$\tilde{u}_n:=u_n-u$ (or a suitable subsequence, not relabeled) as 
$\tilde{u}_{n}=c_n+d_n$ ($c_n=u_{1,n}$ and $d_n=u_{2,n}$ in Lemma~\ref{lem:decloc}). 

We claim that $(Dc_n)$ and $(Du+Dd_n)$ generate $\Lambda_b$ and $\Lambda_i$, respectively. To prove the claim let $\Lambda_1=(\nu_{1,x},\lambda_1,\nu_{1,x}^\infty)\in \genYM$  generated by (a subsequence not relabeled) of $(Dc_n)$, and analogously, let $\Lambda_2=(\nu_{2,x},\lambda_2,\nu_{2,x}^\infty)\in \genYM$ be generated by (a subsequence not relabeled of) $(Du+Dd_n)$. Since all subsequences of $(Dc_n)$ and $(Du+Dd_n)$ have another subsequence generating some generalized Young measure, it suffices to show that
 $$\Lambda_1=\Lambda_b\,\,\, \text{and}\,\,\, \Lambda_2=\Lambda_i.$$
We start by observing that as a consequence of \eqref{Lamdasplit}, the fact that $(D u_n)$ generates $\Lambda$,  Proposition~\ref{prop:fdecloc} and 
Proposition~\ref{prop:asympaddconc}, %and the generating relations for $\Lambda$, $\Lambda_1$ and $\Lambda_2$  
then for every $g\in C(\bar\O)$ and every $v\in \us$, we have that
\begin{equation}\label{Lambdaib12}
\begin{aligned}
	&\langle \langle \Lambda_i, g\otimes v \rangle \rangle
	+\langle \langle \Lambda_b, g\otimes v \rangle \rangle
	 -\int_\O g(x) v(0)\,\md x	\\
	&=~\langle \langle \Lambda, g\otimes v \rangle \rangle~=~\lim_n \int_\O
	g(x) dv(Du+Dd_n+Dc_n)(x)\\
	&=~\lim_n \int_\O g(x) dv(Du+Dd_n)(x)
	+\lim_n\int_\O g(x) dv(Du+Dc_n)(x)-\int_\O g(x) dv(Du)(x)\\
	&=~\lim_n \int_\O g(x) dv(Du+Dd_n)(x)
	+\lim_n \int_\O g(x) dv(Dc_n)(x)-\int_\O g(x) v(0)\,\md x\\
	&=~\langle \langle \Lambda_2, g(x)v \rangle \rangle
	+\langle \langle \Lambda_1, g(x)v \rangle \rangle-\int_\O g(x) v(0)\,\md x.
\end{aligned}
\end{equation}
Moreover,
since $\{c_n\neq 0\}\subset (\partial\Omega)_{\frac{1}{n}}$, it is clear that for every $v\in \us$ and every $g\in C_c(\Omega)$,
$$\langle \langle \Lambda_1, g(x)v \rangle \rangle=
	\lim_n\int_\Omega g(x)dv(Dc_n)(x)=\int_\Omega g(x)v(0)\,\md x.
$$
whence $\nu_{1,x}=\nu_{b,x}=\delta_0$ for a.e.~$x\in\Omega$ and
$\lambda_1\vert_{\O}=\lambda_b\vert_{\O}=0$, meaning that
$\Lambda_1\vert_\O=(\delta_0,0,0)=\Lambda_b\vert_\O$. In view of \eqref{Lambdaib12} with $g\in C_c(\Omega)$, this also implies that
$\Lambda_2\vert_\O=\Lambda_i\vert_\O=\Lambda\vert_\O$.

On the other hand, 
take now $g\in C(\partial\O)$ and extend it (without changing its name) to a function $g\in C(\bar\O)$ by the Tietze theorem, and choose a sequence $(g_k)_{k\in\N}\subset C(\bar\O)$  such that $g_k=g$ on $\partial\O$, $|g_k|\le |g|$ on $\O$ and $\supp g_k\subset (\partial \O)_{\frac{1}{k}}$. 
Since $\Lambda_2$ is generated by $(Du+Dd_n)$, we have for every $k$
\[
\begin{aligned}
	\absb{\langle \langle \Lambda_2, g_k(x)\abs{\cdot} \rangle \rangle}
	&=\absB{\lim_n\int_\O g_k(x)d\abs{Du+Dd_n}(x)}\\
	&\leq C\norm{g_k}_{C(\O)}\sup_n
	\int_{\supp g_k} (d\abs{Du}(x)+d\abs{Dd_n}(x))
	\debaixodaseta {}{k\to +\infty} 0
\end{aligned}
\]
 by dominated convergence and the fact that $(Dd_n)$ does not charge $\partial\Omega$.
Again by dominated convergence, we infer that
$$
	\int_{\partial\O} g(x)\md\lambda_2(x)=
	\lim_k \langle \langle \Lambda_2, g_k(x)\abs{\cdot} \rangle \rangle=0.
$$
Since this holds for arbitrary $g\in C(\partial\O)$, we get
$\lambda_2(\partial\O)=0=\lambda_i(\partial\O)$, meaning that $\Lambda_2\vert_{\partial\O}=\Lambda_i\vert_{\partial\O}$.
In view of \eqref{Lambdaib12}, we conclude that $\Lambda\vert_{\partial\O}=\Lambda_1\vert_{\partial\O}$.

``if'': Assuming that $\Lambda_i\in \GYM$ and $\Lambda_b \in \GYM$ let us see that $\Lambda\in \GYM.$ Consider for that  $(u_n^{(i)}), (u_n^{(b)})\subset W^{1,1}(\O;\R^M)$  two sequences such that
$(\nabla u_n^{(i)})$ generates $\Lambda_i$ and $(\nabla u_n^{(b)})$ generates $\Lambda_b$.
Choosing subsequences if necessary, we may assume that $u_n^{(i)}\debaixodasetafraca {*}{} u^{(b)}$ and  
$u_n^{(i)}\debaixodasetafraca {*}{} u^{(b)}$ in $BV$. Moveover, since $\nu_{b,x}=\delta_0$ for a.e.~$x\in\O$ and
$\lambda_b(\O)=0$, it is clear that $Du^{(b)}=0$ in $\Omega$. Removing a constant if necessary, which does not affect the measure generated by the gradients, we can modify $u_n^{(b)}$ so that $u^{(b)}=0$.
We can decompose both sequences (or suitable subsequences) according to Lemma~\ref{lem:decloc} (again with $J=2$, $K_1=\partial\O$ and $K_2=\bar\O$) into a part $c_n$ purely concentrating at the boundary and an interior remainder $d_n$ which does not charge $\partial\O$:
$$
	u_n^{(i)}-u^{(i)}=c_n^{(i)}+d_n^{(i)}
	\quad\text{and}\quad
	u_n^{(b)}=c_n^{(b)}+d_n^{(b)}
$$
Similarly as in the first part of the proof, $(Dd_n^{(i)}+Du^{(i)})$ still generates $(\Lambda_i)_i=\Lambda_i$, and $(Dc_n^{(b)})$ still generates $(\Lambda_b)_b=\Lambda_b$.
Moreover, Proposition~\ref{prop:fdecloc} applies to the decomposition
\[
	w_n:=c_n^{(b)}+d_n^{(i)},
\]
and together with Proposition~\ref{prop:asympaddconc} we get
\[
\begin{aligned}
	&\lim_n \int_\O g(x) \md v(Du^{(i)}+Dw_n)\\
%	&=~\lim_n \Big(\int_\O g(x) dv(Du^{(i)}+Dd_n^{(i)})
%	+\int_\O g(x) dv(Du^{(i)}+Dc_{n}^{(b)})
%	-\int_{\O} g(x)dv(Du^{(i)})\,\md x\Big)\\
	&=~\lim_n \Big(\int_\O g(x) dv(Du^{(i)}+Dd_n^{(i)})
	+\int_\O g(x) \md v(Dc_n^{(b)})
	-\int_{\O} g(x)v(0)\,\md x\Big)\\
	&=~\langle \langle \Lambda_i, g(x)v \rangle \rangle
	+\langle \langle \Lambda_b, g(x)v \rangle \rangle
	 -\int_\O g(x) v(0)\,\md x=~\langle \langle \Lambda, g\otimes v \rangle \rangle
\end{aligned}
\]
in view of \eqref{Lamdasplit}, which proves that $(Du^{(i)}+Dw_n)$ generates $\Lambda$.
\end{proof}

\section{Necessary conditions at the boundary}\label{sec:nec}

In view of Proposition~\ref{prop:split} (see also the end of Section 
\ref{sec:main}) we will restrict ourselves to investigate $\Lambda_b=(\delta_0,\lambda( \partial\O\cap \cdot),\nu_x^\infty)\in \genYM$ (cf. \eqref{Lamdab}). Necessary conditions for $\Lambda_b$ to be a generalized gradient Young measure
 will be derived with the help of Corollary~\ref{prop:wlsc-1} and Remark~\ref{rem:generator}.

\begin{prop}\label{nec-lb}
If $\Lambda^b\in \GYM$ then there is $\omega_b\subset\partial\O$, $\lambda(\omega_b)=0$ such that
 $$0\leq \langle \nu_x^\infty,v^\infty\rangle$$
 for all $x\in\partial \O\setminus\omega_b$ and all $v \in \us\cap \QSLB(\varrho(x)).$
\end{prop}

\begin{proof}
Suppose that $(\nabla c_n)_n$ generates $\Lambda_b$ for a bounded sequence $(c_n)\subset W^{1,1}(\Omega;\RR^M)$, which in particular imply that $c_n\wstar 0$ in $\mathrm{BV}(\O;\R^M)$.

Given $j\in \N$ let us consider $x_i^j\in \partial \Omega$, $i=1,...,I_j$, such that 
$$\partial \Omega \subset\cup_{i=1}^{I_j}B(x_i^j,1/j).$$
Fix $j\in \N$, let $i\in \{1,...,I_j\}$, and let $v\in \us\cap\QSLB(y)$ for all $y\in \partial \Omega\cap B(x_i^j,1/j)$. Take $ g\in C_0(B(x_i^j,1/j))$, $0\le g$, and extend it by zero to the whole $\bar\O$.  Proposition~\ref{prop:wlsc-1} applied to $f:(x,A)\mapsto g(x)v(A)$ asserts that 
\begin{equation}\label{1554}
	\lim_{n\to\infty}\int_{\O}g(x)v(\nabla c_n(x))\,\md x\ge \int_{\O}g(x)v(0)\,\md x.
\end{equation}
On the other hand, since $(\nabla c_n)_n$ generates $\Lambda_b=(\delta_0,\lambda( \partial\O\cap \cdot),\nu_x^\infty)$,
$$
	\lim_{n\to\infty}\int_{\O}g(x)v(\nabla c_n)\,\md x =\int_{\O}g(x)v(0)\,\md x \\
+\int_{{\bar\O}}\int_{S^{M\times N-1}}
g(x)v^\infty(A)\md \nu_x^{b,\infty}(A)\md\lambda(x)
$$
which from \eqref{1554} implies that
\begin{equation}\label{bola-arg}\int_{\bar\O}g(x)\la \nu_{x}^{b,\infty},v^\infty\ra \, \md\lambda(x)=\int_{\bar\O\cap B(x_i^j,1/j)}g(x)\la \nu_{x}^{b,\infty},v^\infty\ra \, \md\lambda(x)\geq 0.\end{equation}
As $g\in C_0(B(x_i^j,1/j))$, $0\leq g$, is arbitrary, we conclude that 
$$\la \nu_{x}^{b,\infty},v^\infty\ra \geq 0$$
for $x\in  \partial \Omega\cap B(x_i^j,1/j)\setminus E_{i,j,v}$ for some set $E_{i,j,v}\subset \partial \Omega\cap B(x_i^j,1/j)$ with $\lambda(E_{i,j,v})=0.$

Let $\{v_k\}_{k\in\N}$ be a dense subset of $C(S^{M\times N-1}).$ Given $\e>0$ define
$$
	E_{i,j,\e}:=\bigcup \left\{ E_{i,j,v_k(\cdot)+\e |\cdot|}\,:\, k\in \NN~\text{such that}~v_k(\cdot)+\e |\cdot|\in \text{QSLB}(y),\,\,y\in B(x_i^j,1/j)\right\}
$$
and observe that $\lambda(E_{i,j,\e})=0.$ For each $m\in \N$ consider now $\e_{m}=1/m$ and define
$E=\bigcup_{i,j,m}E_{i,j,m}$. Let us see that the claim holds taking $\omega^b=E$.

 For this purpose let $x_0\in \partial \O\setminus\omega^b$ and $v \in \us\cap \QSLB(x_0)$. Given $m$ consider $j$ large enough such that $v(\cdot)+\frac{\e_m}{2} |\cdot|\in \text{QSLB}(y)$ for all $y\in B(x_i^j,1/j)$ where  $i$ is chosen such that $x_0\in B(x_i^j,1/j).$

 Let now $\{v_{k_l}\}$ be a subsequence of $\{v_{k}\}$ such that $|v_{k_l}-v|<\frac{\e_m}{2}$  in $S^{M\times N-1}$. Then $v_{k_l}(\cdot)+{\e_m} |\cdot|\in \text{QSLB}(y)$ for all $y\in B(x_i^j,1/j)$ and thus 
$$\la\nu_{x_0}^{b,\infty},v_{k_l}(\cdot)+\e_m|\cdot|\ra\geq 0.$$
The claim holds by letting $l,m\to +\infty.$

\end{proof}
 
\section{Sufficient conditions at the boundary}\label{sec:suff}

The objective of this section is to prove that condition (iv) in Theorem~\ref{thm:main} allows for a construction of a sequence of gradients generating $\Lambda_b=(\delta_0,\lambda( \partial\O\cap \cdot),\nu_x^\infty)\in \genYM$ (cf.~\eqref{Lamdab}). More precisely we will derive the following result.

\begin{prop}\label{prop:suff-boundary}
Assume  that for $\lambda$-almost all $x\in\partial\O$ we have that $$0\leq \int_{S^{M\times N-1}} v^{\infty}(A)\nu^\infty_{x}(\md A)$$ for all $v\in\us\cap \QSLB(\varrho(x))$.  Then $\Lambda_b\in \GYM$.
\end{prop}

To prove Proposition \ref{prop:suff-boundary} (at the end of this section) we follow ideas developed in \cite{kruzik} and  \cite{fmp},  standardly relying on the Hahn-Banach theorem, allowing us to show that
all $\GYM$-measures that concentrates at the boundary can be approximated 
in the weak*-convergence by a sequence of ``elementary'' measures.

Given a unit vector $\varrho$ in $\R^N$, we start by defining two sets of measures:
$$
A^\varrho:=\{\mu\in\rca(S^{M\times N-1});\ \mu\ge 0\ ,\  \la\mu; v\ra\ge 0\mbox{ for $v\in C(S^{M\times N-1})$ if } v^\infty\in \QSLB(\varrho)\} $$
and 
$$
H^\varrho:=\{\bar\delta_{\varrho,\nabla u}, \,\,\, \ u\in W_0^{1,1}(B;\R^M)\}\ $$
where for all $v\in C(S^{M\times N-1})$, given $u\in W_0^{1,1}(B;\R^M)$,  the measure $\bar\delta_{\varrho,\nabla u}$ is defined by 
$$
\la\bar\delta_{\varrho,\nabla u}, v\ra:=\int_{D_\varrho}v\left(\frac{\nabla u(x)}{|\nabla u(x)|}\right)|\nabla u(x)|\,\md x \ .$$

\noindent Both $H^\varrho$ and $A^\varrho$  are sets of measures on the unit sphere in $R^{M\times N}$ and, in addition, we observe that
\begin{equation}\label{hrho-arho}H^\varrho\subset A^\varrho.\end{equation}  
\noindent  Indeed to see \eqref{hrho-arho} let $\bar\delta_{\varrho,\nabla u}\in H^\varrho$ with $u\in W_0^{1,1}(B;\R^M)$. Given $v\in C(S^{M\times N-1})$ extend it to the whole space as a positively one-homogeneous function (without changing its name)
so that in particular $v=v^{\infty}$. Hence
$$
	\la\bar\delta_{\varrho,\nabla u}, v\ra=\int_{D_\varrho}v^\infty(\nabla u(x))\,\md x
$$
and if $v\in\QSLB(\varrho)$ then (see \ref{qslb-recession}) it holds that $\int_{D_\varrho}v^\infty(\nabla u(x))\,\md x\ge 0$ for all $u\in W_0^{1,1}(B;\R^M)$ i.e. $\bar\delta_{\varrho,\nabla u}\in A^\varrho.$

\begin{lem} \label{lem:hbconvex}
Let $N\ge 2$. Then the set $H^\varrho$ is convex. 
\end{lem}

\begin{proof}
Consider $u_1, u_2\in W_0^{1,1}(B;\R^M)$. We need to show that for any $0\le t\le 1$ the convex combination  
$t\bar\delta_{\varrho,\nabla u_1}+(1-t)\bar\delta_{\varrho,\nabla u_2}\in H^\varrho$. 
Take $x_0\in B\cap\{x\in\R^N;\ \varrho\cdot x=0\}$ such that $|x_0|=1/2$ and define $$\tilde u_1(x):=3^{N-1}u_1(3x) \,\,\, \text{and}\,\,\,
\tilde u_2(x):=3^{N-1}u_1(3(x-x_0)).$$ 
Observe that $\tilde u_1\in W_0^{1,1}(B(0,1/3);\R^M),$ $\tilde u_2\in W_0^{1,1}(B(x_0,1/3);\R^M)$ and then extend these functions by zero to the whole $\R^N$ (without changing its name), so that, in particular, $\tilde u_1,\tilde u_2\in  W_0^{1,1}(B;\R^M)$ and have disjoint supports. Let $$u:=t \tilde u_1+(1-t)\tilde u_2(x).$$ Given $v\in C(S^{M\times N-1})$ extend it, as before, to the whole space as a positively one-homogeneous function. Thus
\begin{align*}
\int_{D_\varrho} v(\nabla u(x))\,\md x &=t\int_{D_\varrho}v(3^{N}\nabla u_1(3x))\,\md x +(1-t)\int_{D_\varrho\cap(B(x_0,1/3)}v(3^{N}\nabla u_1(3(x-x_0))\,\md x\\
&= t\int_{D_\varrho}v(\nabla u_1(y))\,\md y+(1-t)\int_{D_\varrho} v(\nabla u_2(y))\,\md y
\end{align*}
which proves the claim.
\end{proof}

%==================================
%
%{\bf MISSING remark about $N=1$}
%
%==================================

Next result is a standard application of the Hahn-Banach theorem.

\begin{prop}\label{hahn-banach}
The set $A^\varrho$ is the weak* closure of $H^\varrho$. 
\end{prop}

\begin{proof}  
Assume that $A^\varrho\neq \overline{H^\varrho}^{w*}$. As $H^\varrho\subset A^\varrho$ and $A^\varrho$ is clearly closed with respect to the weak*-topology this implies that there exists $\mu_0\in A^\varrho$ such that 
$\mu_0\notin\overline{H^\varrho}^{w*}$. Let $v\in C(\R^{M\times N-1})$ and $a\in\R$ such that 
$$
\la\bar\delta_{\varrho,{\nabla u}}, v\ra\ge a
$$ 
\noindent for all $\bar\delta_{\varrho,{\nabla u}}\in H^\varrho.$
Then extending $v$, as before, to the whole space as a positively one-homogeneous function, and noting that $v=v^\infty$ we get that
$$\inf_{u\in W_0^{1,1}(B;\R^M)} \int_{D_\varrho}v^{\infty}(\nabla u(x))\,\md x\ge a$$
By one-homogeneity this forces $a=0$ so that $\inf_{u\in W_0^{1,1}(B;\R^M)}  \int_{D_\varrho}v^\infty(\nabla u(x))\,\md x=0$. As a consequence $v^\infty\in\QSLB(\varrho)$. Thus if $\mu\in A^\varrho$ in particular $\la\mu,v\ra\ge 0=a$, i.e., no element from $A^\varrho$ can be separated from $H^\varrho$ by a hyperplane defined by $v$ and $a$, which is a contradiction if we apply Hahn-Banach theorem to the sets $H^\varrho$ (recall that is convex) and $\{\mu_0\}.$ 

\end{proof}

\begin{remark}\label{rem:hahn-banach}
By definition, for any given $u\in W^{1,1}_0(B;\RR^M)$, $\bar\delta_{\varrho,\nabla u} \in H^\varrho$
if and only if the generalized Young measure $\big(\delta_0,\lambda:=\norm{\nabla u}_{L^1}\delta_0,\nu_0^\infty:=
\frac{1}{\norm{\nabla u}_{L^1}}\bar\delta_{\varrho,\nabla u}\big)$ on the half-ball $D_\varrho$ is generated
by $(\nabla u_k)$, with $u_k(x):=k^{N-1}u(kx)$. Notice that $(u_k)$ is a rather specific sequence ``purely concentrating'' at $0\in \partial D_\varrho$, supported in balls of radius $\frac{1}{k}$.
In particular, Proposition~\ref{hahn-banach} implies that whenever $\Lambda:=(\delta_0,\delta_0,\nu_x^\infty)$ is a generalized Young measure on $D_\varrho$ such that $\mu:=
\nu_0^\infty\in A^\varrho$, $\Lambda$ is generated by a sequence $(\nabla w_j)$. More precisely, we can find a sequence $(u_j)\in W^{1,1}_0(B;\RR^M)$ such that $\bar\delta_{\varrho,\nabla u_j}\rightharpoonup^* \mu= 
\nu_0^\infty \in A^\varrho$, and then define $w_j(x)=k(j)^{N-1}u_j(k(j)x)$ with $k(j)\to \infty$ fast enough as $j\to\infty$. Also notice that since $\bar\delta_{\varrho,\nabla u_j}\rightharpoonup^* \nu_0^\infty$, we automatically have
$\int_{D_\varrho} \abs{\nabla u_j}\,dx\to 1$ since $\nu_0^\infty$ is probability measure on $S^{MN-1}$, and therefore $\norm{\nabla w_j}_{L^1}\to 1$ on $D_\varrho$.
\end{remark}

\begin{proposition}\label{prop:pointmeasure}
Assume  that  $J\ge 1$, $\Lambda:=(\delta_0, \nu^\infty,\lambda)$ such that $\lambda=\sum_{i=1}^Ja_i\delta_{x_i}$ where $x_i\in\partial\O$ and $a_i>0$ for $1\le i\le J$, and
$\nu^\infty_{x_i}\in A^{\varrho(x_i)}$ for $1\le i\le J$ with $\varrho(x_i)\in\R^N$  the unit outer normal to $\partial\O$ at $x_i$. Then $\Lambda\in \GYM$.
\end{proposition}

\begin{proof}
First notice that $x\mapsto\nu^\infty_x$ is defined $\lambda$-a.e. in $\bar\O$. Assume first that $\nu^\infty_{x_i}\in H^{\varrho(x_i)}$ for all $1\le i\le J$.  Therefore, there is $u_i\in W^{1,1}_0(B;\R^M)$ such that, for every $v^\infty\in C(\RR^{M\times N})$ positively one homogeneous, 
\[
  \int_{D_{\varrho(x_i)}}v^\infty(\nabla u_{i}(y))\,\md y=\int_{S^{M\times N-1}}v^\infty(A)\md\nu_{x_i}^\infty(A)
\]
by definition of $H^{\varrho(x_i)}$.
Consequently,
\begin{align}
\lim_{j\to\infty} a_i\int_{\O\cap B(x_i,1/j)}v^\infty(\nabla u_{i}(j(x-x_i))\,\md x&=
\lim_{j\to\infty} a_i\int_{j(-x_i+\O\cap B(x_i,1/j))}v^\infty(\nabla u_{i}(y))\,\md y\nonumber\\
&=a_i\int_{D_{\varrho(x_i)}}v^\infty(\nabla u_{i}(y))\,\md y
=a_i\int_{S^{M\times N-1}}v^\infty(A)\md\nu_{x_i}^\infty(A)\ .
\end{align}
Consequently, 
$(\sum_{i=1}^J a_i j^{N-1}u_i(j(x-x_i)))_{j\in\N}\subset W^{1,1}(\O;\R^M)$ is a sequence such that the gradients 
$(\sum_{i=1}^J a_i j^{N}\nabla u_i(j(x-x_i)))_{j\in\N}$ generate $\Lambda$.

Assume now that $\nu^\infty_{x_i}\in A^{\varrho(x_i)}$. 
 From Proposition~\ref{hahn-banach},  there is for each fixed $1\le i\le J$, $(u^k_{i})_{k\in\N}\subset W^{1,1}_0(B;\R^M)$  such that for every $v^\infty\in C(S^{M\times N-1})$ positively one homogeneous 
\begin{align}
\lim_{k\to\infty}\lim_{j\to\infty} a_i\int_{\O\cap B(x_i,1/j)}v^\infty(\nabla u^k_{i}(j(x-x_i))\,\md x&=
\lim_{k\to\infty}\lim_{j\to\infty} a_i\int_{j(\O\cap B(x_i,1/j)-x_i)}v^\infty(\nabla u^k_{i}(y)\,\md y\nonumber\\
&=a_i\int_{S^{M\times N-1}}v^\infty(A)\md\nu_{x_i}^\infty(A)\ .
\end{align}
As $C(S^{M\times N-1})$ is separable and $\|\nabla u_i^k\|_{L^1(\O;\R^{M\times N})}\le C$ for all $k\in\N$ the proof is finished by a diagonalization argument. 
\end{proof}

The property of quasi-sublinear growth from below is inevitably connected with the boundary normal. The following proposition shows that suitable rotations allow us to ``average'' along the boundary.

\begin{prop}\label{prop:averaging}
Let $x_0\in\partial\O$. Moreover, let $\{R(x)\}_{x\in \partial\O}\subset {\rm SO}(N)$ be a family of rotation matrices such that 
$x\mapsto R(x)$ is continuous and bounded on a set $\Gamma\subset \partial\O$ 
and for each $x\in\Gamma $, $\varrho(x)=R(x) \varrho(x_0)$.
Given a measurable set $E \subset \Gamma$ such that $\lambda(E)>0$, we define
$\mu_{x_0}=\mu_{x_0,E}\in \rca(S^{M\times N-1})$ as the measure that satisfies
$$
	\int_{S^{M\times N-1}} v(A)\,\md\mu^\infty_{x_0}(A)
	%=\int_{\com\setminus \RR^{M\times N}} v_0(s)\hat{\eta}_{x_0}(ds)
	:=\frac{1}{\lambda(E)}\int_{E} \int_{S^{M\times N-1}}v(A \, R^\top(x))\md\nu^\infty_x( A)\,\md\lambda(x)
$$
for every $v\in C(S^{M\times N-1})$.
If $\{\nu^\infty_x\}_{x\in E}$ satisfies for every $v\in\QSLB(x)$
\begin{align}\label{BGDM}
	&0\leq \int_{S^{M\times N-1}} v(A)\md\nu^\infty_{x}(A)
\end{align}
then for every $v\in\QSLB(x_0)$
\begin{equation} \label{pavcm-2}
	\begin{aligned}
	&0\leq \int_{S^{M\times N-1}} v^\infty(A)\md\mu^\infty_{x_0}(A).
	\end{aligned}
\end{equation} 
\end{prop}

\begin{proof}
For every $v\in\ups$ which is qslb at $x_0$, $A\mapsto v(A R_x^{-1})$ is qslb at $x$ by Lemma~\ref{lem:pqscbrotated}. Hence, \eqref{BGDM}
implies \eqref{pavcm-2} by the definition of $\mu{}_{x_0}$.
\end{proof}

\begin{proof}[Proof of Proposition \ref{prop:suff-boundary}]
For each $n\in\NN$ cover $\RR^N$ with a family of pairwise disjoint cubes of side length $2^{-n}$, translates of $Q_{n,0}:=[0,2^{-n})^N$, and let 
$Q_{n,j}$, $j\in J(n)$, be the collection of those cubes $Q$ in the family that satisfy $\lambda(Q\cap \partial\O)>0$.
Moreover, for each $n$ and each $j\in J(n)$ let $E_{n,j}:=Q_{n,j}\cap \partial\O$,  choose a point $x_{n,j}\in E_{n,j}$, 
and choose a family of rotations $(R_{n,j}(x))_{x\in E_{n,j}}\subset \RR^{N\times N}$ such that
$R_{n,j}(x_{n,j})=\mathbb{I}$, $\varrho(x)=R_{n,j}(x)\varrho(x_{n,j})$ for every $x\in E_{n,j}$, where $\varrho(x)$ denotes the outer normal at $x\in\partial\O$, 
$x\mapsto R_{n,j}(x)$ is continuous on $\overline{E}_{n,j}$ and 
\begin{equation}\label{pwsapprox-2}
	\sup_{j\in J(n)}\sup_{x\in E_{n,j}} \abs{R_{n,j}(x)^{-1}-\mathbb{I}} \underset{n\to\infty}{\To} 0,
\end{equation}
which is possible, at least if $n$ is large enough, since $\partial\O$ is of class $C^1$.
We define for each $n\in \N$
$$
	\theta_n(x):=\sum_{j\in J(n)} \lambda (E_{n,j})\delta_{x_{n,j}}(x)\ , 
$$
and, for every $v\in C(S^{M\times N-1})$ and every $j\in J(n)$,
$$
	\int_{S^{M\times N-1}}v(A)\md\eta^\infty_{n,x}(A)
	%=\int_{\com} v_0(A)\hat{\eta}_{x}(ds)
	%=\int_{\com\setminus \RR^{M\times N}} v_0(A)\hat{\eta}_{x_0}(ds)
	:=%\left\{
	\begin{alignedat}[c]{2}
	&\frac{1}{\lambda(E_{n,j})}\int_{E_{n,j}} \int_{S^{M\times N-1}}v(A\, R^\top(y))\md\nu^\infty_y(A)\,\md\lambda(y)
	&\quad&
	\text{if $x=x_{n,j}$}.
	%\\
	%& \int_{S^{M\times N-1}}v(A)\md\nu^\infty_{x}(A)&\quad&\text{elsewhere.}
	\end{alignedat}%\right.
$$
Here, note that for $x\in \partial\O\setminus \{x_{n,j}\mid j\in J(n)\}$, the definition of 
$\hat{\eta^\infty}_{n,x}$ does not matter since
$\theta_n\big(\partial\O\setminus \{x_{n,j}\mid j\in J(n)\}\big)=0$.
Clearly, $(\delta_0,\eta^\infty_{n,x},\theta_n)\in \GYM$, by Proposition~\ref{prop:pointmeasure} and Proposition~\ref{prop:averaging}.
Finally, observe that, by \eqref{pwsapprox-2} and uniform continuity of continuous functions on compact sets, it holds for any $g\in C(\bar\O)$ and  $v\in C(S^{M\times N-1})$ that,
\begin{align*}
	&\int_{\partial\O} g(x)\int_{S^{M\times N-1}}v(A)\md\eta^\infty_{n,x}(A) \,\md\theta_n(x)
	=\sum_{j\in J(n)} g(x_{n,j})\int_{E_{n,j}}\int_{S^{M\times N-1}}v(A\, R_{n,j}^\top(y))\md\nu^\infty_y(A)\,\md\lambda(y)\\
	&\to \int_{\partial\O}g(y)\int_{S^{M\times N-1}}v(A)\md\nu^\infty_{y}(A)\,\md\lambda (y)\ ,
\end{align*}
as $n\to+\infty$.

As $(\delta_0,\eta^\infty_{n,x},\theta_n)\in \GYM$ for any $n\in\N$, there is for every $n$ a sequence $(u^n_j)_{j\in\N}\subset W^{1,1}(\O;\R^M)$ such that $(\nabla u^n_j)_{j\in\N}$ generates 
$(\delta_0,\eta^\infty_{n,x},\theta_n)$. Due to the fact that $\theta_n(\bar\O)=\lambda(\bar\O)$ there exists $C>0$ such  that $\|\nabla u^n_j\|_{L^1} <C$ for all $j,n$. Moreover, the addition of suitable constants to $u^n_j$ and the Poincar\'{e} inequality gives us $\|u^n_j\|_{W^{1,1}} <C$ for all $j,n$.  Moreover, due to separability of $C(\bar\O)$ and $C(S^{M\times N-1})$ we can extract a 
subsequence of $(u^n_j)_{j,n\in\N}$ such that the gradients generate $\Lambda_b$. Consequently, $\Lambda_b\in\GYM$.  
 
\end{proof}

Applying Theorem~\ref{thm:KriRi} to $\Lambda_i$ defined in \eqref{Lamdai} we see that $\Lambda$ from Theorem~\ref{thm:main} can be written as $\Lambda=\Lambda_i+\Lambda_b$ where $\Lambda_i$ and $\Lambda_b$ are as in Proposition~\ref{prop:split}.  
Altogether, Theorem~\ref{thm:main} is proved.

\section{The Sou\v{c}ek space $W_\mu^1$ and its relation to $BV$ and $\GYM$}\label{sec:soucek}
Nowadays, it is well known that the extension of a functional with linear growth in $W^{1,1}$ subject to Dirichlet boundary conditions leads to an unconstrained functional on $BV$, which, instead of the boundary condition, has an extra term at the boundary penalizing the difference of the ``inner'' trace and the ``true'' boundary value.  
Mostly forgotten by the modern literature covering this topic, J.~Sou\v{c}ek introduced and analyzed the larger space $W_\mu^1$ in \cite{soucek}\footnote{Here, we stick to the original notation $W_\mu^1(\bar\O)$ for the space of \cite{soucek}, although using $BV(\bar\O)$ instead would also make sense. 
The subscript $\mu$ in $W_\mu^1$ is not a parameter; it simply abbreviates ``measure'' and corresponds to the notation $L^1_\mu$ used in \cite{soucek} for the space of Radon measures which we call $\cM$.}
with a similar application in mind. 
Since generalized $BV$-gradient Young measures are actually naturally embedded in $W_\mu^1$, we want to provide a short overview here. Besides, 
for the toy problem featured in the introduction, $W_\mu^1$ is a more natural space than $BV$. As we shall see below, $W_\mu^1(\bar\O)$ is actually isomorphic to
$BV(\O)\times \cM(\partial\O)$. Here, the second component is a measure representing an artificial ``outer'' trace on the boundary.

Just like \cite{soucek}, we assume throughout this section that $\O\subset \RR^N$ is a bounded domain of class $C^1$.
\begin{definition}[Sou\v{c}ek space $W_\mu^1$, {\cite[Definition 1]{soucek}}]\label{def:soucek}
We define $W_\mu^1(\bar\O)$ as the space of those $(u,\alpha)\in L^1(\Omega)\times \cM(\bar{\O};\R^M)$ which are in the
weak$^*$-closure of $W^{1,1}(\O)$ on $\bar\O$ in the sense that for a suitable sequence $(u_n)\in W^{1,1}(\O)$,
\begin{equation*}%\label{weakstarSoucek}
	u_n\to u~~\text{in $L^1(\O)$,~~and}~~
	\int_\O \psi\cdot \nabla u_n\,\md x\to 
	\int_{\bar\O} \psi\cdot d\alpha(x)~~\text{for all $\psi\in C(\bar{\O};\RR^N)$}.
\end{equation*}
The norm of $W_\mu^1$ is given by
\[
	\norm{(u,\alpha)}_{W_\mu^1(\bar\O)}:=\norm{u}_{L^1(\O)}+\sum_{i=1}^N\norm{\alpha_i}_{\cM(\bar\O)},
	\quad\text{where $\alpha=(\alpha_1,\ldots,\alpha_N)$}.
\]
\end{definition}
From the modern perspective, it is clear that any such $u$ is an element of $BV(\O)$, and if we restrict the vector-valued measure $\alpha$ to $\Omega$, it coincides with the $BV$-derivative $Du$. 
In that sense, $W_\mu^1$ can be naturally projected onto $BV$.
However, $\alpha$ carries extra information on $\partial\O$, and it turns out that $\alpha\vert_{\partial \O}$ is exactly the derivative of a jump at $\partial \O$ to some measure-valued ``outer'' trace that in particular
captures the boundary value along sequences satisfying a fixed Dirichlet condition. 

More precisely, Sou\v{c}ek shows that any $(u,\alpha)\in W_\mu^1(\bar\O)$ has two different traces on $\partial \O$: The first one is the \emph{outer
trace} $\beta\in \cM(\partial \O)$ characterized by the version of Green's formula derived in \cite[Theorem 1]{soucek} (with $\varrho$ denoting the outer normal on $\partial\O$):
\begin{equation}\label{GreenSoucek}
	\int_{\partial\O} \varphi \varrho \,d\beta
	=\int_\O u \nabla\varphi\,\md x
	+\int_{\bar\O}\varphi d\alpha~~\text{for every $\varphi\in C^1(\bar\O)$}.
\end{equation}
The outer trace is weakly$^*$-continuous with respect to weak$^*$ convergence of $(u,\alpha)$ in $\cM(\bar\O)\times \cM(\bar\O;\RR^N)$ \cite[Theorem 2]{soucek}, in strong contrast to the situation in $BV$. Here, observe that the kind of behavior of a sequence that has to be ruled out in $BV$ to ensure continuity of the trace (usually by imposing strict convergence), namely, a jump moving to or growing at the boundary, is not lost in the weak$^*$-limit in $W_\mu^1(\O)$ as it appears in $\alpha\vert_{\partial\O}$. In addition, in \cite[Definition 4]{soucek} the \emph{inner trace} $\beta^0\in \cM(\partial \O)$ of $(u,\alpha)\in W_\mu^1$ is defined as the outer trace of $(u,\bar\alpha)$, where $\bar\alpha$ is obtained from $\alpha$ by dropping the contribution on $\partial\O$, i.e., $\bar\alpha\vert_\O:=\alpha\vert_\O=Du$ and $\bar\alpha\vert_{\partial\O}:=0$. 
Here, $(u,\bar\alpha)\in W_\mu^1$ due to \cite[Theorem 8]{soucek} (alternatively, one can use the fact that the $BV$-function $u$ is the limit of a strictly convergent sequence $(u_n)\subset W^{1,1}$ and observe that 
strict convergence implies that $\nabla u_n\rightharpoonup^* \bar\alpha$ in $\cM(\bar\O)$, whence $(u,\bar\alpha)\in W_\mu^1$).
Due to \eqref{GreenSoucek} and the fact that $\alpha=Du$ on $\O$, $\beta^0$ coincides with trace of $u$ in the sense of $BV$.
As a consequence (or, alternatively, due to \cite[Theorem 10]{soucek}), the inner trace in $W_\mu^1$ -- here first introduced as a Radon measure on $\partial\O$ -- is actually absolutely continuous with respect to the $(N-1)$-dimensional Hausdorff measure on $\partial\O$, and thus it can be interpreted as an element of $L^1(\partial\O)$.
\begin{remark}
Just like $\alpha\vert_{\partial \O}$, the outer trace $\beta$ can have singular contributions, possibly even charging single points. In particular, this occurs whenever $(u,\alpha)$ is obtained as the weak$^{*}$ limit of a sequence whose gradients concentrate at a single point at the boundary, as in Example~\ref{ex:pointconc}. As a side effect, the results presented here cannot be easily extended to domains with Lipschitz boundary even with more refined tools from geometric measure theory: a normal $\varrho$ defined $\cH^{N-1}$-a.e.~is simply not good enough to write \eqref{GreenSoucek} if $\beta$ (and $\alpha\vert_{\partial\O}$) charges a set with Hausdorff dimension below $N-1$.
\end{remark}
Another simple but useful observation is that the remainder of the projection of $(u,\alpha)\in  W_\mu^1$ ``onto $BV$'' given by $(u,\bar\alpha)\in  W_\mu^1$, 
%recall that $\abs{\bar\alpha}(\partial\O)=0$ and $\bar\alpha=\alpha=Du$ on $\O$).
the \emph{side} $(0,\chi_{\partial\O}\alpha)= (u,\alpha)-(u,\bar\alpha)$ (cf.~\cite[Definition 4]{soucek}) also is an element of $W_\mu^1$ \cite[Theorem 8]{soucek}. 
In particular, the side of $(u,\alpha)$  satisfies \eqref{GreenSoucek}, which reduces to
\begin{equation}\label{GreenSoucekside}
	\int_{\partial\O} \varphi \varrho \,d(\beta-\beta^0)
	=\int_{\partial\O}\varphi d\alpha~~\text{for every $\varphi\in C^1(\bar\O)$},
\end{equation}
where $\beta$ and $\beta^0$ denote the outer and inner trace of $(u,\alpha)$, respectively. A straightforward but nonetheless remarkable consequence (also pointed out in \cite[Theorem 9]{soucek}) is that for every $(u,\alpha)\in W_\mu^1$, $\abs{\alpha}=\abs{\beta-\beta^0}$ on $\partial\O$ and
\begin{equation}\label{siderankone}
	\frac{d\alpha^s}{d\abs{\alpha^s}}(x)=\frac{d(\beta-\beta^0)^s}{d\abs{(\beta-\beta^0)^s}}(x)\varrho(x)
	~~\text{for $\abs{\alpha}$-a.e.~$x\in \partial\O$.}
\end{equation}
\begin{remark}\label{rem:rank1boundary}
If $u$ is vector- instead of scalar-valued, \eqref{siderankone} translates into saying that the then matrix-valued measure $\alpha^s$ is a matrix of rank one $\abs{\alpha}$-a.e.~on~$\partial\O$, more precisely, of the form $a(x)\otimes\varrho(x)$ with a suitable vector-valued $a(x)$.
\end{remark}

Similar to $BV$, $W_\mu^1$ can also be characterized using \eqref{GreenSoucek}:
\begin{thm}[{\cite[Section 4]{soucek}}; see also {\cite[Definition 2]{soucek}}]\label{thm:SoucekGreen}
For any $(u,\alpha)\in W^{1,\mu}$, \eqref{GreenSoucek} is satisfies with the outer trace $\beta$ of $(u,\alpha)$.
Conversely, for all $(u,\alpha,\beta)\in L^1(\O)\times \cM(\bar\O;\RR^N) \times \cM(\bar\O)$ such that 
\eqref{GreenSoucek} holds, $(u,\alpha)\in W^{1,\mu}$ and $\beta$ is its outer trace.
\end{thm}
All of the above has obvious extensions for the vector-valued case, i.e., for the space $W^1_\mu(\bar\O;\RR^M)\subset L^1(\O;\RR^M)\times \cM(\bar\O;\RR^{M\times N})$ with $M>1$.
Since we defined $\GYM$ in that case, we will stick to this setting from now on.
The relationship of $BV$, $W_\mu^1$ and $\GYM$ can be summarized as follows:
\begin{thm}[$BV$ as a subset of $W_\mu^1$]\label{thm:SoucekBV}
We have the following chain of (natural) continuous embeddings and isomorphisms:
\[
BV(\O) ~\hookrightarrow~ BV(\O;\RR^M)\times L^1(\partial\O;\RR^M) 
~\hookrightarrow~ BV(\O;\RR^M)\times \cM(\partial\O;\RR^M) ~\cong~ W_\mu^1(\bar\O;\R^M). 
\]
\end{thm}
Before giving a proof, let us discuss how $\GYM$ fits into this picture. Roughly speaking, $\GYM$ exclusively contains information on gradients, unlike the other spaces appearing in Theorem~\ref{thm:SoucekBV}: for each $\Lambda\in\GYM$, the underlying deformation $u$ is only determined up to a constant vector ($\Omega$ being connected). 
For a meaningful comparison, we therefore have to remove these constants, or simply drop $u$ in the pair $(u,\alpha)\in W_\mu^1$ (recall that $Du$ can be reconstructed from $\alpha$ anyway). 
\begin{thm}[$W_\mu^1$ as a subset of $\GYM$]\label{thm:SoucekGYM}
Let
\[
	GW_\mu^1:=\{ \alpha \mid (u,\alpha)\in W_\mu^1~~\text{for a suitable}~u\in L^1 \}.
\]
%be the set of all ``gradient'' measures appearing in $W_\mu^1$.
\begin{itemize}
\item[(i)]
For every $\alpha\in GW_\mu^1(\bar\O;\RR^M)$, 
we have 
\[
	\Lambda=\Lambda(\alpha):=
	\big(\delta_{\nabla u(x)},\abs{\alpha^s},\delta_{\frac{d\alpha^s}{d\abs{\alpha^s}}(x)}\big)
	\in \GYM(\bar\O;\RR^M),
\]
where $\nabla u$ denotes the densitiy of the absolutely continuous part of $\alpha$ (or $Du=\alpha\vert_\O$) with respect to $\cL^N$, and $\alpha^s$ is the singular part of $\alpha$
with polar decomposition $d\alpha^s(x)=\frac{d\alpha^s}{d\abs{\alpha^s}}(x)d\abs{\alpha^s}(x)$. 
\item[(ii)]
Conversely, for every $\Lambda\in \GYM(\bar\O;\RR^M)$, we have an 
associated pair $(u(\Lambda),\alpha(\Lambda))\in W_\mu^1(\bar\O;\RR^M)$, where
the underlying deformation $u=u(\Lambda)$
is given by Theorem~\ref{thm:main} (uniquely determined up to a constant in $\RR^M$) and 
$\alpha=\alpha(\Lambda)\in GW_\mu^1(\bar\O;\RR^M)$
is the center of mass of $\Lambda$, i.e.,
\[
	d\alpha(x)=\langle \nu_x,\id \rangle\,\md x+\langle \nu^\infty_x,\id \rangle\,\md\lambda(x).
\]
\end{itemize}
Moreover, $\alpha(\Lambda(\gamma))=\gamma$ for every $\gamma\in GW_\mu^1$.
%the map $\alpha\mapsto \Lambda(\alpha)$ is one-to-one,
%both $\alpha\mapsto \Lambda(\alpha)$ and $\Lambda\mapsto \alpha(\Lambda)$
%are continuous with respect to the weak$^*$ topologies in $\cM$ and $\genYM$, respectively, 
\end{thm}
\begin{proof}[Proof of Theorem~\ref{thm:SoucekBV}]
The first two embeddings being trivial, we only have to show that $BV(\O)\times \cM(\partial\O) ~\cong~ W_\mu^1(\bar\O)$. 
Given $(u,\alpha)\in W_\mu^1(\bar\O)$, we have $(u,\beta)\in BV(\O)\times \cM(\partial\O)$, where
$\beta$ is the outer trace of $(u,\alpha)$, and the operator $W_\mu^1(\bar\O)\to BV(\O)\times \cM(\partial\O)$, $(u,\alpha)\mapsto (u,\beta)$, clearly is linear and continuous. It is also one-to-one, because
$\alpha\vert_\O=Du$ (the derivative of $u\in BV$) and $d\alpha\vert_{\partial\O}=\varrho(x)d(\beta-\beta^0)$, where $\beta^0$ the inner trace of $(u,\alpha)$ which is fully determined by $u$ and $\alpha\vert_\O=Du$. 
%(and coincides with the trace of $u$ in the sense of $BV$).

It remains to show that $(u,\alpha)\mapsto (u,\beta)$, $W_\mu^1(\bar\O)\to BV(\O)\times \cM(\partial\O)$ is onto. Let $(u,\beta)\in BV(\O)\times \cM(\partial\O)$. We define $\alpha\in \cM(\bar\O;\RR^N)$ by
$\alpha\vert_\O:=Du$ and $d\alpha\vert_{\partial\O}(x):=\rho(x)d(\beta-\beta^0)$, with the trace $\beta^0\in L^1(\partial\O)\subset \cM(\partial\O)$ of $u$ in the sense of $BV$. Clearly, we have
\begin{equation*}
	\int_{\partial\O} \varphi \varrho \,d\beta^0
	=\int_\O u \nabla\varphi\,\md x
	+\int_{\O}\varphi dDu~~\text{for every $\varphi\in C^1(\bar\O)$},
\end{equation*}
i.e, \eqref{GreenSoucek} holds for $(u,\bar\alpha,\beta^0)=,\beta^0)$. In addition, \eqref{GreenSoucekside} holds by definition of $\alpha\vert_{\partial\O}$, and, equivalently, \eqref{GreenSoucek} holds for $(0,\chi_{\partial \O} \alpha,\beta-\beta^0)$. 
Adding up, we infer that \eqref{GreenSoucek} holds for $(u,\alpha,\beta)=(u,\bar\alpha,\beta^0)+(0,\chi_{\partial \O} \alpha,\beta-\beta^0)$. By Theorem~\ref{thm:SoucekGreen}, we conclude that $(u,\alpha)\in W_\mu^1$ and that $\beta$ is its outer trace.
\end{proof}
\begin{proof}[Proof of Theorem~\ref{thm:SoucekGYM}]
``Moreover'': This is clear from the definition of $\Lambda(\alpha)$ and $\alpha(\Lambda)$.

\noindent (i): Given $\alpha\in GW_\mu^1$, we define $Du:=\alpha\vert_\O$ (which really is the derivative of a function $u\in BV$ since $\alpha\in GW_\mu^1$),
and claim that 
\[
	\Lambda:=\big(\delta_{\nabla u(x)},\abs{\alpha^s},\delta_{\frac{d\alpha^s}{d\abs{\alpha^s}}(x)}\big)
	\in \GYM(\bar\O;\RR^M)
\]
where $\nabla u=\frac{d\alpha}{d\cL^N}=\frac{dDu}{d\cL^N}$. It suffices to check 
conditions (i)--(iv) in Theorem~\ref{thm:main}. 
Clearly, (i) is valid since $\langle\langle \Lambda, 1\otimes \abs{\cdot} \rangle\rangle=\int_\O \abs{\nabla u}\,\md x+\int_{\bar\O} d\alpha^s<\infty$. Both (ii) and (iii) hold with equality by definition of $\Lambda$ and $Du$. Finally, (iv) is a consequence of the fact that 
$\frac{d\alpha^s}{d\abs{\alpha^s}}(x)$ is a matrix of rank $1$ of the form $a(x)\otimes \varrho(x)$ for $\abs{\alpha^s}$ a.e.~$x\in\partial\O$ (see Remark~\ref{rem:rank1boundary}) combined with Lemma~\ref{lem:slb-along-rank1}.

\noindent (ii): Given $\Lambda \in \GYM(\bar\O;\RR^M)$, let $(u_n)\subset W^{1,1}(\O;\RR^M)$ be a sequence such that $(\nabla u_n)$ generates $\Lambda$.
In particular, we have that
\[
	\int_\O g(x)\nabla u_n(x) \,\md x\to \int_\O \langle \nu_x,\id \rangle\,\md x + \int_{\bar\O}
	\langle \nu^\infty_x,\id \rangle\,\md\lambda(x)
\]
for all $g\in C(\bar\O)$, 
This means that $\nabla u_n\rightharpoonup^* \alpha$ in $\cM(\bar\O;\RR^{M\times N})$, where
\[
	d\alpha(x):=\langle \nu_x,\id \rangle\,\md x + \langle \nu^\infty_x,\id \rangle\,\md\lambda(x),
\]
whence $\alpha\in GW^1_\mu$. Now let $u\in BV$ be the function given by Theorem~\ref{thm:main}. Since the identity $\id:\RR^{M\times N}\to \RR^{M\times N}$ is linear, $\pm \id$ is quasiconvex and its own (generalized) recession function. Thus, the inequalities (ii) and (iii) of Theorem~\ref{thm:main} hold for $f=\id$ as well as for $f=-\id$, which implies equality in both cases. Hence,
\[
	\nabla u(x)\md x=\langle \nu_x,\id \rangle\,\md x+\langle \nu^\infty_x,\id \rangle\,\frac{\md\lambda}{d\cL^N}(x)\md x
	\quad\text{and}\quad
	\md Du^s(x)= \langle \nu^\infty_x,\id \rangle\,\md\lambda^s(x).
\]
Combined, we get $\md Du(x)=\nabla u(x)\md x+\md Du^s(x)=d\alpha(x)$, and
consequently, $(u,\alpha)\in W^1_\mu(\O;\R^M)$.
\end{proof}

\section{Traces for $\GYM$}\label{sec:GYMtraces}

In order to handle functionals with boundary terms, it is helpful to have a notion of trace for elements of $\GYM$.
As always, we assume that $\Omega\subset\RR^N$ is a bounded domain with a boundary of class $C^1$. Moreover, let
$\Gamma\subset \partial\O$ be open in $\partial\O$. Below, Lebesgue spaces on $\Gamma$ and other subsets of $\partial\O$ are always understood with respect to the surface measure $\cH^{N-1}$.

\begin{defn}[inner and outer trace in $\GYM$ on $\Gamma$]
Let $\Lambda\in \GYM$, and let $(u,\alpha)$ be the associated element of the Sou\v{c}ek space $W^1_\mu$ in the sense of Theorem~\ref{thm:SoucekGYM} (ii), i.e., $u$ is the underlying deformation of $\Lambda$ (defined up to a constant), and $\alpha$ is the center of mass of $\Lambda$.
We define the \emph{inner trace} $T_i\Lambda$ of $\Lambda$ in $L^1(\Gamma;\RR^M)$ 
as the restriction to $\Gamma$ of the inner trace of $(u,\alpha)\in W^1_\mu$ on $\partial\O$ (which coincides with the trace of $u$ in the sense of $BV$).
Analogously, the outer trace $T_o\Lambda\in \cM(\Gamma;\RR^M)$ of $\Lambda$ on $\Gamma$ is defined as 
the restriction to $\Gamma$ of the outer trace of $(u,\alpha)\in W^1_\mu$.
\end{defn}
\begin{remark}~
\begin{enumerate}
\item[(i)] Since the underlying deformation $u\in BV$ of $\Lambda$ is only defined up to a constant in $\RR^M$, 
the same holds for $T_i\Lambda$ and $T_o\Lambda$, with the same constant for all three. 
In our applications below, this does not really play a role, though, because either the constant does not matter, or one of the three is given and the constant is therefore fixed, anyway.
\item[(ii)] If we express the derivative $\alpha$ of $(u,\alpha)\in W^1_\mu$, the element of the Sou\v{c}ek space associated to $\Lambda$, as a generalized Young measure as in Theorem~\ref{thm:SoucekGYM} (i), we get $\bar\Lambda$, the center of mass of $\Lambda$. Hence $T_i\Lambda=T_i\bar\Lambda$ and $T_o\Lambda=T_o\bar\Lambda$.
\item[(iii)] The inner trace of $(u,\alpha)\in W^1_\mu$ was defined as the outer trace of $(u,\bar\alpha)$ with $\bar\alpha:=\chi_\Omega\alpha$. Accordingly, we have $T_i\Lambda=T_o \Lambda^{i}$, where  $\Lambda^{i}=\chi_\Omega\Lambda\in \GYM$ is the interior part of $\Lambda$ in the sense of Section~\ref{sec:separate}.
\item[(iv)] The explicit formula \eqref{siderankone} for the difference between outer and inner trace in the Sou\v{c}ek space translates as follows to the trace difference in $\GYM$:
\begin{align}\label{GYMtracedifference}
  dT_o\Lambda(x)-T_i\Lambda(x)d\cH^{N-1}(x)=\langle \nu_{x}^\infty , id \rangle \varrho(x)d\lambda(x),\quad x\in \Gamma,
\end{align}
where $\varrho$ is the outer normal on the boundary.
\end{enumerate}
\end{remark}
It is tempting to think that the traces of $\Lambda\in \GYM$ should be more general objects than the traces of $(u,\alpha)\in W^1_\mu$, say, (generalized) Young measures on $\Gamma$ instead of elements of $L^1(\Gamma;\RR^M)$ or $\cM(\Gamma;\RR^M)$. However, as we shall see below, the inner trace enjoys a rather strong compactness property
in $L^1(\Gamma;\RR^M)$, so that a larger class is pointless, while for the outer trace, the information contained in $\Lambda$ simply does not suffice to uniquely (up to a constant) determine the trace as a (generalized) Young measure on $\Gamma$.

The following theorem describes the link between traces in $\GYM$ and the limits of traces of associated generating sequences.
\begin{thm}\label{thm:traces}
Suppose that $(u_k)\subset BV$ is a bounded sequence generating $\Lambda=(\nu_x,\lambda,\nu_x^\infty)$, and $u_k\rightharpoonup^*u$ in $BV$.
Then $Tu_k$ (the trace of $u_k$ on $\Gamma$ in the sense of $BV$) satisfies $Tu_k\rightharpoonup^* T_o\Lambda$ in $\cM(\Gamma;\RR^M)$. Moreover, if $\lambda(\overline{\Gamma})=0$, then $Tu_k\to T_i\Lambda=T_o\Lambda=Tu$ strongly in $L^1$.
\end{thm}
\begin{proof}
$Tu_k\rightharpoonup^* T_o\Lambda$: This immediately follows from the weak$^*$-continuity of the outer trace in the Sou\v{c}ek space.

``Moreover'': We only give a proof for the case where $\Gamma$ is contained in a hyperplane in $\RR^N$. The general case can be recovered with standard tools in such a context, namely a decomposition of unity, maps locally straightening the boundary, and suitable error estimates. Rotating and shifting if necessary, we may even assume that $\Gamma=\{0\}\times \Gamma'$ for a suitable open $\Gamma'\subset \RR^{N-1}$, and the outer outer normal on $\Gamma\subset \partial\O$ is given by $-e_1=(-1,0,\ldots,0)$.

Since $\lambda(\overline{\Gamma})=0$, $T_o \Lambda=T_i \Lambda$. It therefore suffices to show that $(Tu_k)$ is Cauchy in $L^1(\Lambda;\RR^M)$. Using the density of $C^1(\bar\O;\RR^M)$ in $W^{1,1}(\O;\RR^M)$, we may assume that $(u_k)\subset C^1(\bar\O;\RR^M)$, while we still have that $u_k\rightharpoonup^*u$ in $BV$ and $(\nabla u_k)$ generates $\Lambda=(\nu_x,\lambda,\nu_x^\infty)\in \GYM$.
The basic observation now is the following: For any $w \in C^1(\bar\O;\RR^M)$, $x'\in \Gamma'$ and $t>0$ small enough so that $(0,t)\times \Gamma'\subset \O$, we have
\begin{equation}\label{tGYMt-basic}
	\frac{1}{t}\int_0^t \abs{w(x_1,x')-w(0,x')}\,dx_1 = \frac{1}{t}\int_0^t \absB{ \int_0^{x_1} \partial_1 w(s,x')\,ds}\,dx_1
	\leq \int_0^t \abs{\partial_1 w(s,x')}\,ds
\end{equation}
Due to \eqref{tGYMt-basic},
\[
\begin{aligned}
	&\int_{\Gamma'} \abs{u_j(0,x')-u_k(0,x')}\,dx'\\
	&=\int_{\Gamma'}\frac{1}{t}\int_0^t \abs{u_j(0,x')-u_k(0,x')}\,dx_1\,dx'\\
	&=\int_{\Gamma'}\frac{1}{t}\int_0^t \abs{u_j(0,x')-u_j(x_1,x')+u_k(x_1,x')-u_k(0,x')
	+u_j(x_1,x')-u_k(x_1,x')}\,dx_1\,dx'\\
	&\leq \int_{\Gamma'}\frac{1}{t}\int_0^t \abs{u_j(0,x')-u_j(x_1,x')}+\abs{u_k(x_1,x')-u_k(0,x')}
	+\abs{u_j(x_1,x')-u_k(x_1,x')}\,dx_1\,dx'\\
	&\leq \int_{\Gamma'}\frac{1}{t}\int_0^t \abs{u_j(0,x')-u_j(x_1,x')}+\abs{u_k(x_1,x')-u_k(0,x')}
	+\abs{u_j(x_1,x')-u_k(x_1,x')}\,dx_1\,dx'\\
	&\leq 
	\int_{(0,t)\times\Gamma'} \abs{\partial_1 u_j(x)}+\abs{\partial_1 u_k(x)} \,dx
	+
	\frac{1}{t}\int_{(0,1)\times \Gamma'} \abs{u_j(x)-u_k(x)}\,dx\\
	&=:A(t,j,k)+\frac{1}{t}B(j,k)
\end{aligned}
\]
Since $(u_k)$ is Cauchy in $L^1(\O;\RR^M)$, for any $t>0$ there is $N_0(t)$ such that $B(j,k)\leq t^2$ if $j,k\geq N_0(t)$. 
In addition, if $\lambda(\Gamma)=0$,
\[
	\lim_{t\to 0} \lim_{k,j\to \infty} A(t,j,k)=
	 2\int_{\Gamma} \langle  \nu^\infty_x,\abs{\cdot e_1} \rangle \,d\lambda(x)=0.
\]
Hence, $(Tu_k)$ is Cauchy in $L^1(\Gamma';\RR^M)$. 
\end{proof}

The weak$^*$ convergence we get for a sequence of outer traces in $\cM(\Gamma;\RR^M)$ is not sufficient to pass to the limit in boundary terms like an integral functional on $\Gamma$ with a nonlinear integrand. 
Unfortunately, there seems to be no easy way to fix that, because $\GYM$ does not carry enough information to define, say, a (generalized) Young-measure outer trace:
\begin{example}\label{ex:GYMtraceproblem}
Let $\Gamma':=(0,1)$, $(0,1)\times \Gamma'\subset \Omega\subset \RR^2$, and $\Gamma:=\{0\}\times \Gamma'\subset \partial\O$ (therefore, the outer normal is $\varrho:=(-1,0)$ on $\Gamma$), and take two sequences $(u_k)$ and $(v_k)$ of scalar $BV$ functions on $\Omega$ purely concentrating on $\Gamma$, with $u_k=v_0=0$ on $\Omega\setminus (0,1)^2$ and the following additional properties for $(x_1,x')\in (0,1)^2$:
\begin{itemize}
\item $u_k(x_1,x')=0$ for $x_1\geq \frac{1}{k^2}$, $\nabla u_k(x_1,x')=k^2 \varrho$ for $x_1<\frac{1}{k^2}$. The trace is constant $u_k(0,x')=1$ on $\Gamma'$, generating the Young measure $\delta_1$ 
on $\Gamma'$ as $k\to\infty$ (as well as the generalized Young measure $(\delta_1,0,-)$).
\item $v_k(x_1,x')=0$ for $x_1\geq \frac{1}{k^2}$, $\nabla v_k(x_1,x)=2k^2 \varrho$ for $(x_1,x')\in (0,\frac{1}{k^2})\times (0,\frac{1}{k})$, $\nabla v_k(x_1,x')=0$ for $(x_1,x')\in (0,\frac{1}{k^2})\times (\frac{1}{k},\frac{2}{k})$, repeated $2/k$-periodically in $x'$. Accordingly, the trace $v_k(0,\cdot)$ alternates between $2$ and $0$. In the limit as $k\to\infty$, the sequence of traces generates the Young measure $\frac{1}{2}\delta_2+\frac{1}{2}\delta_0$ on $\Gamma'$ (as well as the generalized Young measure $(\frac{1}{2}\delta_2+\frac{1}{2}\delta_0,0,-)$).
\end{itemize}
Despite their different Young measure ``traces'', both sequences generate the \emph{same} generalized gradient Young measure on $\bar\Omega$, namely $(\delta_0,\chi_{\Gamma}\cH^1,\delta_{\varrho})$: Concentrating at $\Gamma$ with either gradient $-k^2 \varrho(x')$ on a set of measure $\frac{1}{k^2}$ or gradient $-2k^2 \varrho(x')$ on a set of measure $\frac{1}{2k^2}$ give the same limit in $\GYM$; the interior jumps of $v_k$ on $(x_1,x')\in (0,\frac{1}{k})\times (\Gamma'\cap\{\frac{j}{k}\mid j\in\ZZ\})$ are negligible, because their jump height is bounded (at most $2$), and they are supported on a jump set whose surface measure is $\frac{k+1}{k^2}\to 0$.
\end{example}

\section{Relaxation of functionals defined on $W^{1,1}$}\label{sec:relaxation}

In this section, we state and prove a relaxation result in terms of generalized Young measure suitable for functionals like the prototype in the introduction and higher-dimensional analogues.
To include nonlinear boundary terms as in those examples, it would be ideal if we were able to characterize under which circumstances a pair $(\Lambda,\beta)$, a generalized gradient Young measure and its outer trace,
can be generated by a sequence in $W^{1,1}$ such that the traces converge to $\beta$ \emph{strongly}, not just weakly$^*$ in $\cM$, because this is crucial for the construction of recovery sequences even if the boundary integrand is strictly convex. But so far, we are unable to provide such a result, not even for the case $\beta\in L^1$. In particular, Example~\ref{ex:GYMtraceproblem} shows that one can easily choose the ``wrong'' recovery sequence once oscillating traces on the boundary are possible. Of course, none of these issues are relevant for functionals with linear boundary integrals, or no boundary integrals at all. 

Below, we do allow nonlinear boundary terms, but avoid the issue with recovery sequences by imposing a few extra, possibly technical conditions on the integrand that make oscillations on the boundary energetically unfavorable, as well as concentrations at the boundary other that a simple jump growing or moving there. Some extra notation is needed for that, which we will introduce next. 

\subsection{Prerequisites}
For $\Lambda\in\GYM(\bar\O;\RR^M)$, we define the \emph{first moment} (sometimes also called \emph{center of mass}) $\la\Lambda,\id\ra\in \cM(\bar\O;\RR^{M\times N})$ of $\Lambda$ as 
\begin{align}
  \md\la\Lambda,\id\ra(x) := \la \nu_x,\id \ra\,\md x+\la \nu^\infty_x,\id \ra\,\md\lambda(x),
\end{align}
i.e., for any $g\in C(\bar\O)$
$$\int_{\bar\O}g(x)\md\la\Lambda,\id\ra(x)= \int_\O g(x)\int_{\R^{M\times N}} A\md\nu_x(A)\,\md x+\int_{
\bar\O}g(x)\int_{S^{M\times N-1}} A\md\nu^\infty_x(A)\,\md \lambda(x)\ .
$$
Alternatively, the first moment can also be written as a generalized Young measure $\bar\Lambda=(\bar\nu_x,\bar\lambda,\bar\nu_x^\infty)$ where $\bar\nu_x$ and $\bar\nu_x^\infty$ consistently are Dirac masses:
\[
 \bar\nu_x:=\delta_{\la \nu_x,\id \ra},~~
 d\bar\lambda(x):=\abs{\la \nu^\infty_x,\id \ra}d\lambda(x),~~
 \bar\nu_x^\infty:=\delta_{\frac{\la \nu_x^\infty,\id \ra}{\abs{\la \nu^\infty_x,\id \ra}}}.
\]
\begin{remark}\label{rem:centerofmass-boundary-rank1}
Due to Theorem~\ref{thm:SoucekGYM} (ii) and Remark~\ref{rem:rank1boundary}, $\la \nu^\infty_x,\id \ra$ is a rank one matrix of the form $\varrho(x) \otimes a(x)$ for $\lambda$-a.e.~$x\in\partial\O$, where $\varrho(x)$ denotes the outer normal as usual and $a:\partial\O\to\RR^M$ is a suitable function.
As a consequence, $\Lambda\in\GYM$ implies that $\bar\Lambda\in \GYM$, by Theorem~\ref{thm:main} and Lemma~\ref{lem:slb-along-rank1}.
\end{remark}

Below, we will rely on a variant of quasiconvexity at the boundary, a Jensen-type inequality we only need for recession functions:
\begin{defn}[$JQCB(\varrho)$]\label{def:jqcb}
Given a unit vector $\varrho\in S^{N-1}$ and $v^\infty:\RR^{M\times N}\to \RR$ continuous and positively $1$-homogeneous, we say that $v^\infty\in JQCB(\varrho)$ if
\begin{align}
  v^\infty \Big(\int_{D_\varrho}\nabla\varphi\,\md y \Big)\leq \int_{D_\varrho} v^\infty(\nabla\varphi)\,\md y
  &\quad\text{for every $\varphi\in W_0^{1,\infty}(B;\RR^M)$,}
  \label{jqcb}
\end{align}
\end{defn}
\begin{remark}
Clearly, all convex positively $1$-homogeneous functions belong to $JQCB(\varrho)$.
Moreover, if $v^\infty$ is quasiconvex at the boundary with respect to the normal $\varrho$ as defined in
\cite{sprenger,mielke-sprenger} (originally, the notion is due to \cite{bama}), i.e.,
for every $A$ there exists a vector $b=b(A)\in \RR^M$ such that
\begin{align}
  \int_{D_{\varrho}} [v^\infty(A+\nabla\varphi)-v^\infty(A)]\,dy\geq 
  \int_{B\cap \{y\mid y\cdot \varrho=0\}} b\cdot \varphi\,\md \cH^{N-1}(y)
  &\quad\text{for every $\varphi\in W_0^{1,\infty}(B;\RR^M)$,}
  \label{qcbrec}
\end{align}
then $v^\infty\in JQCB(\varrho)$. However, the proof of the latter is not entirely trivial, mainly because we cannot add non-constant affine functions to $\varphi$ in \eqref{qcbrec} without leaving the class of admissible test functions.
Details are given in Appendix~\ref{sec:qcb-jqcb}.
\end{remark}

\begin{lemma}\label{lem:GYM_and_qcb}
Let $\Lambda=(\nu_x,\lambda,\nu_x^\infty)\in \GYM(\bar\O;\RR^M)$. Then for $\lambda$-a.e.~$x\in\partial\O$,
\[
  v^\infty(\la \nu^\infty_x,\id\ra)\leq \la \nu^\infty_x, v^\infty \ra
  \quad\text{for every $v\in \us$ which is qcb at $x$.}
\]
\end{lemma}
\begin{proof}
Fix $x_0\in \partial\O$ and let $(\varphi_k)\subset W^{1,1}(B;\RR^M)$ denote a bounded sequence such that $\supp \varphi_k\subset \frac{1}{k}B$ and $(\nabla\varphi_k)$ generates $(\delta_0,\delta_0,\nu_{x_0}^\infty)$ on $D_\varrho$ with the outer normal $\varrho$ at $x_0$ (see Remark~\ref{rem:hahn-banach}). Mollifying if necessary, we may even assume that $(\varphi_k)\subset C_c^\infty(B;\RR^M)\subset W_0^{1,\infty}(B;\RR^M)$. Since $v^\infty$ is qcb, we have
\[
  v^\infty \Big( \int_{D_\varrho} \nabla\varphi_k \,dy \Big)\leq \int_{D_\varrho}  v^\infty(\nabla\varphi_k)\,dy.
\]
This implies the assertion in the limit as $k\to\infty$ because $v^\infty$ is continuous,
$\int_{D_\varrho}  \nabla\varphi_k\,dy\to \la \nu^\infty_{x_0},\id\ra$ and 
 $\int_{D_\varrho}  v^\infty(\nabla\varphi_k)\,dy\to \la \nu^\infty_{x_0}, v^\infty \ra$.
 
\end{proof}

\subsection{Relaxation}
Take $\O\in\R^N$ a bounded domain with a boundary of class $C^1$. In addition, let $f:\bar\O\times\R^{M\times N}\to \RR$ and $g:\partial\O\times\R^{M}\to \RR$ be continuous, such that:
\begin{align*}
&\text{$c^{-1}(-1+|A|)\le f(x,A)\le c(1+|A|)$ for all $x\in\bar\O$, $A\in\R^{M\times N}$;} \tag{f:1}\label{f1}\\
&\text{$f$ has a recession function in the sense of \eqref{recessionx};} %for all $x\in\bar\O$;} 
\tag{f:2}\label{f2}\\
&\text{$c^{-1}(-1+|\mu|)\le g(x,\mu)\le c(1+|\mu|)$ for all $x\in \Gamma$, $\mu\in\R^{M}$;} \tag{g:1}\label{g1} \\
&\text{$g$ has a recession function in the sense of \eqref{recessionx}.} \tag{g:2}\label{g2} \\
\end{align*}
Here, $c$ is a suitable constant. We also define
\[
\begin{aligned}
  \Gamma_R& :=\{x\in \partial\O\setminus \mid g(x,\cdot)\neq 0\}, \\
  \Gamma_N& :=\partial\O\setminus \Gamma_R.
\end{aligned}
\]
By definition, $\Gamma_R$ and $\Gamma_N$ are disjoint and cover $\partial\O$. 
Moreover, $\Gamma_N$ is closed, while $\Gamma_R$ is open in $\partial \O$.
Admissible functions remain unconstrained both on $\Gamma_R$ and $\Gamma_N$, which for critical points formally leads to natural boundary conditions of (nonlinear) Robin and Neumann type, respectively.

We are looking at the following problem: 
\begin{align}\label{pb:minimum}
\begin{aligned}
  \text{ minimize } F(u):=\int_\O f(x,\nabla u(x))\,\md x  
  +\int_{\partial\O} g(x,u)\,\ \md\mathcal{H}^{N-1}(x)&\\
  \text{subject to }   u\in \calA_C&,
\end{aligned}
\end{align}
where $C>0$ is given and
\begin{align*}
\mathcal{A}_C:= \{u\in W^{1,1}(\O;\R^M);\  \|u\|_{L^1(\partial\O;\R^M)}\le C\ ,\  \|\nabla u\|_{L^1(\O;\R^{M\times N})}\le C\}\ .
\end{align*}
Here, $u$ as a function on $\partial \O$ or subsets thereof is understood in the sense of trace, and Lebesgue spaces on $\partial \O$ or its subsets are understood with respect to the measure $\cH^{N-1}$.
One particular example is the toy problem \eqref{toyproblem} mentioned in the introduction. 

A natural extension of \eqref{pb:minimum} to $W^{1}_\mu$ is given by
\begin{align}\label{pb:minextended}
\begin{aligned}
  \text{ minimize } \bar{F}(u,\alpha):=\int_\O df(x,\alpha)
  +\int_{\Gamma} dg(x,\beta) &\\
  \text{subject to } (u,\alpha)\in\overline{\calA}_C&.
\end{aligned}
\end{align}
where $\beta$ is the outer trace of $(u,\alpha)$ on $\partial\O$, $\beta^s$ its singular part and $\frac{d\beta}{d\calH^{N-1}}$ the density of its absolutely continuous part with respect to $\calH^{N-1}$,
\begin{align*} 
  dg(x,\beta):=g\Big(x,\frac{d\beta}{d\calH^{N-1}}(x)\Big)d\calH^{N-1}(x)+
  g^{\infty}\Big(x,\frac{d\beta^s}{d\abs{\beta^s}}(x)\Big)\,d\abs{\beta^s}(x)
\end{align*}
and
\begin{align*}
  \overline{\mathcal{A}}_C:= \{(u,\alpha)\in W^{1}_\mu(\O;\R^M);\  \norm{\beta}_{\cM(\partial\O;\R^M)}\le C\ ,\
  \norm{\alpha}_{\cM(\bar\O;\R^{M\times N})}\le C\}\
\end{align*}

Even more general, for $\Lambda\in \GYM$ with an outer trace $\beta$ on $\partial\O$ we may consider:
\begin{align}\label{pb:relaxation}
\begin{aligned}
  \text{ minimize } \hat{F}(\Lambda,\beta):=\la\la\Lambda, f\ra\ra +\int_{\Gamma} dg(x,\beta)&\\
  \text{ subject to }   (\Lambda,\beta) \in \hat\calA_C&,
\end{aligned}
\end{align}
where
\begin{align*}
\begin{alignedat}{2}
  & \hat\calA_C&&:=\{ (\Lambda,\beta)\in \GYM_C(\bar\O;\RR^M)\times \cM_C(\partial \O;\RR^M) \,:\, 
  \text{$\Lambda=\bar\Lambda$ on $\Gamma$, $\beta$ is an outer trace of $\Lambda$} \},\\
  & \GYM_C&&:=\{\Lambda\in \GYM\,:\, \la\la \Lambda, 1\otimes \abs{\cdot} \ra\ra \leq C\},~~
  \cM_C:=\{\beta\in \cM \,:\, \norm{\beta}_{\cM}\le C\}
\end{alignedat}
\end{align*}

%\begin{remark}
%As an outer trace of $\Lambda$ (or, equivalently, $\bar\Lambda$), $\beta$ also is the outer trace of some $(u,\alpha)\in W^{1}_\mu$ such that $\alpha=\bar\Lambda$. Given $\Lambda$, the latter condition determines both $u$ and $\beta$ up to a constant (the same for both).
%\end{remark}

\begin{theorem}\label{thm:relax}
Assume that in addition to \eqref{f1}, \eqref{f2}, \eqref{g1} and \eqref{g2}, 
$g(z,\cdot)$ is convex, $g^\infty(z,\cdot)\geq 0$ and $f^\infty(z,\cdot)\in QSLB(\varrho(z))\cap JQCB(\varrho(z))$ 
for every $z\in \Gamma_R$.
Then problem \eqref{pb:relaxation} is the relaxation of both \eqref{pb:minimum} and \eqref{pb:minextended}, i.e., $\inf F=\inf \bar{F}=\min\hat F$. 
In addition, the following holds:
\begin{enumerate}
\item[(i)] Every minimizing sequence of either \eqref{pb:minimum} or \eqref{pb:minextended} contains a subsequence (not relabeled) which generates a minimizer of \eqref{pb:relaxation} in the sense that 
\begin{align}\label{thmrelax-convergence}
  \text{$(\nabla u_k)$ generates $\Lambda\in \GYM$, and $u_k\rightharpoonup^* \beta$ in $\cM(\partial\O;\RR^M)$},
\end{align}
where on $\partial\Omega$, $u_k\in L^1(\partial\O;\RR^M)\subset \cM(\partial\O;\RR^M)$ is understood in the sense of traces in $W^{1,1}$.
\item[(ii)] Conversely, for every $(\Lambda,\beta)=((\nu_x,\lambda,\nu_x^\infty),\beta)\in \hat\calA_C$ 
(arbitrary), we have $\hat{F}(\Lambda,\beta)\geq \hat{F}(\tilde{\Lambda},\tilde{\beta})$, where
$(\tilde{\Lambda},\tilde{\beta})=((\tilde{\nu}_x,\tilde{\lambda},\tilde{\nu}_x^\infty),\tilde{\beta})\in \hat\calA_C$
is given by
\[
\begin{aligned}
  &\tilde{\Lambda}:=\left\{\begin{array}{ll}
  (\nu_x,\lambda,\nu_x^\infty)~&\text{on}~\O\cup \Gamma_N,\\
  \Big(-,\abs{\la \nu_x^\infty,\id\ra}\frac{d\lambda}{d\calH^{N-1}}\calH^{N-1},
  \delta_{\frac{\la \nu_x^\infty,\id\ra}{\abs{\la \nu_x^\infty,\id\ra}}}\Big)~&\text{on}~\Gamma_R
%  \\
%  \Big(\delta_0,\frac{d\lambda}{d\calH^{N-1}}\calH^{N-1},
%  \nu_x^\infty\Big)~~&\text{on}~\partial\O\setminus\Gamma,
  \end{array}\right.\\
  %\quad
   &\tilde{\beta}:=\left\{\begin{array}{ll}
    \beta &~~\text{on}~\Gamma_N,\\
    \frac{d\beta}{d\calH^{N-1}}(x)\calH^{N-1}&~~\text{on}~\Gamma_R.
  \end{array}\right.
\end{aligned}
\]
Moreover, if $(\Lambda,\beta)$ is a minimizer of \eqref{pb:relaxation} (and thus also  $(\tilde{\Lambda},\tilde{\beta})$), then $(\tilde{\Lambda},\tilde{\beta})$ is generated by a suitable minimizing sequence $(u_k)$ of \eqref{pb:minimum}, again in the sense of \eqref{thmrelax-convergence}.
\end{enumerate}
\end{theorem}
\begin{remark}
By definition, $\tilde{\Lambda}$ satisfies
$\tilde{\Lambda}=\Lambda$ on $\O\cup \Gamma_N$ and 
$\tilde{\Lambda}=\bar\Lambda$ on $\Gamma_R$, where $\bar\Lambda$ denotes the first moment of $\Lambda$ as before.
Analogous to Remark~\ref{rem:centerofmass-boundary-rank1}, it is clear that $\tilde{\Lambda}$ defined in that way really is an element of $\GYM$ for each $\Lambda\in \GYM$. 
\end{remark}
\begin{remark}
Above, the sets $\calA_C,\tilde\calA_C,\hat\calA_C$ of admissible functions/gradient Young measures and their (outer) traces have a built-in bound, determined by the constant $C$. In this way, we avoid having to talk about coercivity conditions and a priori bounds. However, if a suitable a priori bound does hold, then $C>0$ can of course be chosen large enough so that all minimizing sequences will be admissible up to a finite number of members. In this case, Theorem~\ref{thm:relax} also holds without the artificial bound given by $C$.
\end{remark}
\begin{proof}[Proof of Theorem~\ref{thm:relax}]
Clearly, $\calA_C\subset \overline{\calA}_C \subset \hat\calA_C$
(embedded in the sense of Theorem~\ref{thm:SoucekBV} and Theorem~\ref{thm:SoucekGYM}), $\bar{F}=F$ on $\calA_C$ and 
$\hat{F}=\bar{F}$ on $\overline{\calA}_C$.  As a consequence,
\begin{align}\label{thmrelax-lowerbound0}
  \inf_{\calA_C} F\geq \inf_{\overline{\calA}_C} \bar{F}\geq \inf_{\hat{\calA}_C} \hat{F}.
\end{align}
In addition, $\hat\calA_C$ is sequentially compact with respect the convergence in \eqref{thmrelax-convergence}, i.e., weak$^*$ convergence in $\GYM(\bar\O;\RR^M)\times \cM(\partial \O;\RR^M)$. Also observe that the first term in $\hat{F}$ is a continuous function of $\Lambda$ with respect to weak$^*$ convergence in $\GYM$. 
Since $g$ is convex and continuous, just like $g^\infty$, the corresponding second term in $\hat{F}$ is a lower semicontinuous function of $\beta$ with respect to weak$^*$ convergence in $\cM$. 
In particular, we get that $\hat{F}$ attains its minimum in $\hat\calA_C$.
It now suffices to show (i) and (ii), and we start with the latter.

\noindent {\bf (ii):} Let $(\Lambda,\beta) \in \hat\calA_C$. 
By definition, we also have $(\tilde\Lambda,\tilde\beta) \in \hat\calA_C$.
We first show that
\begin{align}\label{thmrelax-lowerbound1}
  \hat{F}(\Lambda,\beta) ~\geq~  \hat{F}(\tilde\Lambda,\tilde\beta)
\end{align}
Since $g\geq 0$ on $\Gamma_R$, all singular contributions on $\Gamma_R$ in the boundary term are non-negative, and we get
\[
  \int_{\partial \O} \md g(x,\beta)\geq \int_{\partial \O} \md g(x,\tilde{\beta}).
\]
In addition, since $f^\infty(z,\cdot)\in QSLB(\varrho(z))$ for every $z\in \Gamma_R$, 
we have $\la \nu_x^\infty,f^\infty(x,\cdot) \ra \geq 0$ for $\lambda$-a.e.~$x\in \Gamma_R$,
which allows us to drop any singular contribution with respect to $\calH^{N-1}$ on $\Gamma_R$ in the 
first integral in \eqref{pb:relaxation}:
\[
  \int_{\Gamma_R} \la \nu_x^\infty,f^\infty(x,\cdot) \ra \md\lambda \geq 
  \int_{\Gamma_R} \la \nu_x^\infty,f^\infty(x,\cdot) \ra \frac{\md\lambda}{\md \cH^{N-1}}(x)
  \md \cH^{N-1}.
\]
Last but not least, due to the fact that $f^\infty(z,\cdot)\in JQCB(\varrho(z))$ for every $z\in \Gamma_R$ and Lemma~\ref{lem:GYM_and_qcb},
\[
\begin{aligned}
  \int_{\Gamma_R} \la \nu_x^\infty,f^\infty(x,\cdot) \ra \frac{\md\lambda}{\md \cH^{N-1}}(x)
  \md \cH^{N-1}
  &\geq \int_{\Gamma_R} f^\infty(x,\la \nu_x^\infty, \id\ra) \frac{\md\lambda}{\md \cH^{N-1}}(x)
  \md \cH^{N-1}\\
  &= \int_{\Gamma_R} f^\infty\Big(x,\frac{\la \nu_x^\infty, \id\ra}{\abs{\la \nu_x^\infty, \id\ra}}\Big) 
  \md\tilde{\lambda}
\end{aligned}
\]
concluding the proof of \eqref{thmrelax-lowerbound1}. 

In the rest of the proof, we repeatedly split generalized (gradient) Young measures into their boundary and interior parts, respectively, as defined in Section~\ref{sec:separate}.
Furthermore, we will also use a slightly more general modification of a given generalized (gradient) Young measure, namely,
for $\Lambda=(\nu_x,\lambda,\nu_x^\infty)\in \genYM$ and a Borel set $\Gamma\subset \bar\O$, with a slight abuse of notation we define
\[
  \chi_{\Gamma}\Lambda:=
  (\chi_{\Gamma}(x)\nu_x+\chi_{\bar\O\setminus\Gamma}\delta_0,
  \lambda(\Gamma\cap\cdot), \chi_{\Gamma}(x)\nu_x^\infty),
\]
i.e., $\chi_{\Gamma}\Lambda=\Lambda$ on $\Gamma$, whereas $\chi_{\Gamma}\Lambda$ is the ``trivial'' generalized Young measure $(\delta_0,0,-)$ on $\bar\O\setminus \Gamma$. In this sense, we have in particular that $\Lambda^i=\chi_\O \Lambda$ and
$\Lambda^b=\chi_{\partial \O} \Lambda$.

Now let $(\Lambda,\beta)$ be a minimizer of  \eqref{pb:relaxation}. Since on $\Gamma_R$, $\tilde{\Lambda}$ coincides with the 
center of mass of $\Lambda$,
$\chi_{\Gamma_R}\tilde{\Lambda}$ can be interpreted as an element $(0,\tilde\alpha)\in W_\mu^1$ according to Theorem~\ref{thm:SoucekGYM} (ii), where $d\tilde\alpha(x)=\chi_{\Gamma_R}(x)\la \nu_x^\infty, \id\ra d\lambda(x)$.
%Let $\tilde\beta$ denote the outer trace of $(0,\tilde\alpha)\in W_\mu^1$, so that $T_o (\chi_{\Gamma_R}\tilde{\Lambda})=\chi_{\Gamma_R}\tilde{beta}$. 
Since $\tilde{\lambda}|_{\Gamma_R}$ is absolutely continuous with respect to $\calH^{N-1}$ by definition, the outer trace $T_o (\chi_{\Gamma_R}\tilde{\Lambda})$ is also absolutely continuous with respect to $\calH^{N-1}$. Therefore, $T_o (\chi_{\Gamma_R}\tilde{\Lambda})$ is even the trace of some $w\in BV(\RR^N\setminus \Omega;\RR^M)$ on $\partial\O$, and the extension $\chi_{\RR^N\setminus \Omega}w$ of $w$ to $\RR^N$ is an element of $BV(\RR^N;\RR^M)$. 
Consequently, we can choose a bounded sequence $(w_k)\subset W^{1,1}(\RR^N;\RR^M)$ such that $w_k=w$ on $\RR^N\setminus \Omega$, $w_k\to \chi_{\RR^N\setminus \Omega}w$ strictly in $BV(\RR^N;\RR^M)$, $w_k=0$ on $\Omega\setminus (\partial \Omega)_{\frac{1}{k}}$. In particular, $(\nabla w_k)\subset L^1(\Omega;\RR^{M\times N})$ generates $\chi_{\Gamma_R}\tilde{\Lambda}^b=\chi_{\Gamma_R}\tilde{\Lambda}$. This also implies that 
$\normn{w_k}_{W^{1,1}}\to\normn{\chi_{\Gamma_R}\tilde{\Lambda}^b}_{\genYM}$, and modifying $w_k$ by multiplying with a suitable sequence of constants converging to $1$, we may assume w.l.o.g.~that $\normn{w_k}_{W^{1,1}}\leq \normn{\chi_{\Gamma_R}\tilde{\Lambda}^b}_{\genYM}$ for all $k$.
 
Now take two other bounded sequences $(y_k),(z_k)\in W^{1,1}$ such that $(\nabla y_k)$ generates $\chi_{\Gamma_N}\tilde{\Lambda}^b=\chi_{\Gamma_N}\Lambda$
and
$(\nabla z_k)$ generates $\tilde{\Lambda}^i=\Lambda^i=\chi_{\O}\Lambda$.
In addition, w.l.o.g., we may assume that
$\norm{y_k}_{W^{1,1}}\leq \la\la \chi_{\Gamma_N}\tilde{\Lambda}^b,1\otimes \abs{\cdot}\ra\ra$ and
$\norm{z_k}_{W^{1,1}}\leq \la\la \tilde{\Lambda}^i,1\otimes \abs{\cdot}\ra\ra$. 
For $u_k:=z_k+y_k+w_k$, we get that $(\nabla u_k)$ generates $\tilde\Lambda$ due to Lemma~\ref{lem:gYMadd}. 
As a consequence, $u_k\rightharpoonup^* \tilde{\beta}=T_o \tilde{\Lambda}$ on $\partial\O$.

To see that $(u_k)$ is a minimizing sequence for \eqref{pb:minimum}, observe the following:
$z_k\to T_i \tilde{\Lambda}$ strongly in $L^1$ on $\partial\O$ by Theorem~\ref{thm:traces},
$y_k\to 0$ strongly in $L^1$ on every compact $K\subset \Gamma_R$,
and $w_k=\tilde{\gamma}:=T_o (\chi_{\Gamma_R}\tilde{\Lambda}) =\chi_{\Gamma_R} (T_o \tilde{\Lambda}-T_i \tilde{\Lambda})$ is a constant sequence on $\Gamma_R$.
In particular, besides weak$^*$ convergence in $\cM$, we also have
\[
  \text{$u_k\to T_i \tilde{\Lambda} %+ T_o(\chi_{\Gamma_N}\tilde{\Lambda})
  + \tilde{\gamma}=\tilde{\beta}$ strongly in $L^1_{\rm loc}(\Gamma_R;\RR^M)$,}
\]
and since both $g$ and $g^\infty$ are continuous with $g(x,\cdot)=0$ for $x\in \Gamma_N=\partial\O\setminus \Gamma_R$, we conclude that
\[
  \int_{\Gamma_R} g(x',u_k(x'))\,d\cH^{N-1}(x')\to \int_{\Gamma_R} g(x',\tilde{\beta}(x'))\,d\cH^{N-1}(x')
\]
Altogether,
we have $F(u_k)\to \hat{F}(\tilde{\Lambda},\tilde{\beta})$ as $k\to \infty$.

\noindent {\bf (i):} Let $(u_k)$ be a minimizing sequence of \eqref{pb:minimum}, which generates $(\Lambda_0,\beta_0)\in \hat\calA_C$.
Due to lower semicontinuity, $\hat F(\Lambda_0,\beta_0)\leq \inf_{\calA_C} F=\lim F(u_k)$.
Picking a minimizer $(\tilde{\Lambda},\tilde{\beta})$ as in (ii), we get another minimizing sequence $(w_k)$
such that $\inf_{\calA_C} F=\lim F(w_k)=\hat F(\tilde{\Lambda},\tilde{\beta})=\min_{\hat\calA_C}\hat F$.
As a consequence, $\min_{\hat\calA_C}\hat F \leq \hat F(\Lambda_0,\beta_0) \leq \inf_{\calA_C} F = \min_{\hat\calA_C}\hat F$, and we have equality everywhere.
\end{proof}

\appendix

\section{Proof of Proposition~\ref{prop:wlsc-1}}\label{Appendix-A}

We start with the following lemma.

\begin{lem}
\label{prop:lscbndconc}(cf.~\cite[Proposition 5.5]{BKK2013})
Assume that $f=g\otimes v$ with $(g,v)\in C(\bar\O)\times \us$ and that $F$ is given by \eqref{BV-functional}. Let $\cU\subset \mathrm{BV}(\Omega;\RR^N)$ be an additively closed set,
%let $\cU$ be an additively closed subset of $\mathrm{BV}(\Omega;\RR^n)$, let $v\in BV(\Omega;\RR^M)$ 
and suppose that $F:\cU\to \RR$ is bounded from below. Then 
$F$ is weak*-lower semicontinuous along all sequences
 $(c_n)_{n \in \mathbb{N}}\subset \cU$ that are bounded in $\mathrm{BV}(\Omega;\RR^N)$
and in addition satisfy that
%$\liminf F(u_n)\geq F(0)$ for every bounded sequence such that 
$S_n:=\{c_n\neq 0\}\cup \supp \abs{D c_n}\subset (\partial \Omega)_{r_n}$ for a decreasing sequence $r_n\searrow 0$,  which, in particular, imply that $c_n\wstar 0$ in $\mathrm{BV}(\Omega;\RR^N)$.  %(in particular, $u_n\rightharpoonup 0$ weakly$^*$ in $BV$).

\end{lem}

{\bf Proof of Proposition~\ref{prop:wlsc-1}.} For fixed $\eps>0$, we cover $\partial \Omega$ by the following collection of balls:
\begin{equation}
\partial \Omega \subset \bigcup_{x \in \partial \Omega} \bigcup_{\delta \leq \tilde{\delta}(x,\eps)} B_\delta (x),
\label{coverOm}
\end{equation}
where $\tilde{\delta}(x,\eps)$ is any such radius for which \eqref{def:qslb} holds; here we recall that if this condition holds with the ball of radius $\tilde{\delta}(x,\eps)$ it also holds for any ball of smaller radius.

Further, since $\partial \Omega$ is a compact we can chose from the cover in \eqref{coverOm} a finite subcover
$$
\partial \Omega \subset \bigcup_{j=1}^J B_{\delta_j} (x_j)
$$
with the radii bounded from below, i.e. $\delta_j \geq \delta_0$ for some $\delta_0 = \delta_0(\eps)$. In fact, since $B_{\delta_j} (x_j)$ are open and the collection is finite, we may still find $\alpha > 0$ so that balls of the radii $\delta_j - \alpha$ still cover $\partial \Omega$; i.e.
$$
\partial \Omega \subset \bigcup_{j=1}^J \overline{B_{\delta_j-\alpha} (x_j)}.
$$

Let us now apply the local decomposition Lemma \ref{lem:decloc} to the sequence $(c_n)_{n \in \mathbb{N}}$ with the compact sets
\begin{align*}
K_1 &= \overline{B_{\delta_1-\alpha} (x_1)} \cap \overline{\Omega}, \\
&\vdots \\
K_J &= \overline{B_{\delta_J-\alpha} (x_J)} \cap \bar{\Omega}, \\
K_{J+1} &= \overline{\Omega \setminus \bigcup_{j=1}^J \overline{B_{\delta_j-\alpha} (x_j)}};
\end{align*}
so we can write %(after a possible choice of sub-sequence and for $n$ large enough)
\begin{equation}
c_n = c_{1,n} + c_{2,n} + \ldots +c_{J+1,n},
\label{Part-Boundary}
\end{equation}
where $c_{j,n}$ are supported in $B_{\delta_j} (x_j)$ for $j=1\ldots J$ is  and $c_{J+1,n}$ is supported in $\Omega$. Notice that we need $n$ large enough dependening on $\alpha$ and $\delta_0$ in order to fulfill these requirements. Moreover, $c_{1,n} \ldots c_{J,n}$ retain the property of the original sequence to be concentrating on the boundary while $c_{J+1,n}=0$ for large $n$,  and so
\begin{equation}
F(0) = \lim_{n \to \infty} F(c_{J+1,n}).
\label{Cj+1}
\end{equation}

Further, we define the auxiliary functionals
\begin{align*}
	&G_j(v):=\int_{\Omega\cap B_{\delta_j} (x_j)} 
	f(x,\nabla u(x))+ \eps|\nabla u(x)|\,\md x\ ,\\
	&u\in \cU_j:=\Big\{v\in \mathrm{BV}(\Omega\cap B_{\delta_j} (x_j);\RR^n) \text{ with } v=0 \text{ near } \partial B_{\delta_j} (x_j)\Big\}.
\end{align*}
Each is bounded from below due to the given quasisublinear growth from below \eqref{def:qslb}. Therefore, they are lower semicontinuous along sequences purely concentrating on the boundary due to Lemma~\ref{prop:lscbndconc}; in particular, $G_j$ is lower semicontinuous along $(c_{j,n})$ (note that indeed $(c_{j,n})$ vanishes near $\partial B_{\delta_j} (x_j)$). As a consequence, 
%again choosing a subsequence if necessary,
\begin{align*}
	\lim_{n \to \infty} F(c_{j,n})-F(0)&=
	\lim_{n \to \infty} G_j(c_{j,n})-G(0)-\eps  \int_{\Omega\cap B_{\delta_j} (x_j)}|\nabla c_n(x)|\,\md x\nonumber \\
	&\geq -\eps \lim_{n \to \infty} \int_{\Omega\cap B_{\delta_j} (x_j)}|\nabla c_n(x)|\,\md x.
\end{align*}
By \eqref{Cj+1} and Proposition \ref{prop:fdecloc} (which applies to $F$ as well as to $u\mapsto \abs{\nabla u}$), the sum over $j$ yields that
\begin{align}
	\lim_{n \to \infty} F(c_n)-F(0)
	&\geq -\eps \lim_{n \to \infty} \int_\O|\nabla c_n(x)|\,\md x,
	\label{Cjn}
\end{align}
As $\varepsilon>0$ is arbitrary and $(c_n)_n$ is bounded in $W^{1,1}(\O;\R^M)$ the claim follows.
%\section{DiPerna-Majda measures}\label{diperna-majda}
 %\begin{definition}[DiPerna-Majda measure generated by a sequence $\{U_n\}$]
 %Let $\{U_n\}\subset L^{1}(\O;\R^{M\times N})$,  be a bounded sequence, and let $\sigma \in \mathcal{M}(\overline{\O})$ and 
 %$\hat{\nu}\in L_{w}^{\infty}
 %(\overline{\O},\sigma;{\rm rca(\beta_{\mathcal R}\R^{M\times N}}))$.  The pair $(\sigma,\hat{\nu})$ is said to be a DiPerna Majda measure generated by $\{U_n\}$ if
 %$$\int_{\O}\phi(x)v(U_n(x))\, \md x \debaixodaseta {}{n\to \infty} \int_{\overline\O} \phi(x)\langle \hat{\nu}_{x}, v_{0}\rangle \md\sigma(x) $$
%\noindent for every $\phi\in C(\overline{\O})$ and every $v_{0}\in \mathcal R$ with $v(\cdot)=v_{0}(\cdot)(1+|\cdot|^p)$.
 %\end{definition}

The following appendix links together generalized Young measures and the so-called DiPerna-Majda measures \cite{diperna-majda} which allow for much more general treatment of concentration effects and which were used by the last two authors in many occasions \cite{chk,ifmk,mkak,KK2011,kruzik}. We confine ourselve to a special case of $p=1$ and  of the compactification of $\R^{M\times N}$ by the sphere when one can show a one-to-one correspondence of the DiPerna-Majda and generalized Young measures. More general ntegrands are treated with  this tool  in \cite{KK2011,k-r-dm, r}.

\section{Relation of generalized Young measures to DiPerna-Majda measures}
A different but equivalent description of oscillation and concentration effects in $L^p$-bounded sequences can be reached by means of DiPerna-Majda measures \cite{diperna-majda, k-r-dm, k-r-control}. However, DiPerna-Majda measures allow for much more general test functions if we replace  the sphere used in the definition of $\genYM$
by a larger set. We refer to \cite{r} for a comprehensive treatment. 

In what follows we will  work   mostly  with a particular compactification of $\R^{M\times N}$, namely, with the compactification by the sphere. We will consider
the following ring  of continuous bounded functions
\begin{eqnarray}\label{spherecomp}
\mathcal{S}&:=&\left\{^{^{^{^{^{}}}}} v_0\in C(\R^{M\times N}):\mbox{ there exist } c\in\R\ ,\ v_{0,0}\in C_0(\R^{M\times N}),\mbox{ and }  v_{0,1}\in C(S^{M\times N-1}) \mbox{ s.t. }\right.\nonumber\\
&  & \left.
 v_0(A) = c+ v_{0,0}(A)+v_{0,1}\left(\frac{A}{|A|}\right)
\frac{|A|}{1+|A|}\mbox { if $A\ne 0$ and }  v_0(0)=v_{0,0}(0)+c\right\}\ ,
\end{eqnarray}  
where $S^{M\times N-1}$ denotes the $(MN-1)$-dimensional unit sphere in $\R^{M\times N}$. Then $\beta_\mathcal{S}\R^{M\times N}$ is homeomorphic to the unit ball $\overline{B(0,1)}\subset \R^{M\times N}$ via the mapping $d:\R^{M\times N}\to B(0,1)$, $d(A):=A/(1+|A|)$ for all $A\in\R^{M\times N}$. Note that $d(\R^{M\times N})$ is dense in $\overline{B(0,1)}$. Notice also that, in view of \eqref{class},

$$\us=\{v:\R^{M\times\N}\to\R;\ v/(1+|\cdot|)\in\mathcal{S}\}\ .$$ 

%For any $v\in\us$  there exists a continuous and positively one-homogeneous function $v^\infty:\R^{M\times N}\to\R$ (i.e. $v^\infty(\alpha A)=\alpha v^\infty(A)$ for all $\alpha\ge 0 $ and $ A\in\R^M$) such that 
%\be\label{recessionf}
%\lim_{|A|\to\infty}\frac{v(A)-v^\infty(A)}{|A|}=0\ .
%\ee 
%
%Indeed, if $v_0$ is as in (\ref{spherecomp}) and $v=v_0(1+|\cdot|)$  then set
%$$v^\infty(A):=\left(c+v_{0,1}\left(\frac{A}{|A|}\right)\right)|A|\mbox{ for $A\in\R^{M\times N}\setminus\{0\}$.} $$
%By continuity we define $v^\infty(0):=0$. It is easy to see that $v^\infty$ satisfies (\ref{recessionf}).
%Such $v^\infty$ is called the {\it recession function} of $v$.

Let $\sigma\in\rca(\bar\O)$ be a  positive Radon measure on a bounded domain $\O\subset\R^N$. A
mapping $\hat\nu:x\mapsto \hat\nu_x$ belongs to the
space $L^{\infty}_{\rm w}(\bar\O,\sigma;\rca(\beta_\mathcal{S}\R^{M\times N}))$ if it is weakly*  $\sigma$-measurable (i.e., for any $v_0\in C_0(\R^{M\times N})$, the mapping
$\bar\O\to\R:x\mapsto\int_{\beta_\mathcal{S}\R^{M\times N}} v_0(A)d\hat\nu_x(A)$ is $\sigma$-measurable in
the usual sense). If additionally
$\hat\nu_x\in\prca(\beta_\mathcal{S}\R^{M\times N})$ for $\sigma$-a.a. $x\in\bar\O$
 the collection $\{\hat\nu_x\}_{x\in\bar{\O}}$ is the so-called
Young measure on $(\bar\O,\sigma)$ \cite{y}, see also
\cite{ball3,r,valadier,warga}.

DiPerna and Majda \cite{diperna-majda} shown that having a bounded
sequence in $(Y_k)_{k\in\N}\subset L^1(\O;\R^{M\times N})$ with $1\le p<+\infty$ and
$\O$ an open domain in $\R^N$, there exists its subsequence
(denoted by the same indices), a positive Radon measure
$\sigma\in\rca(\bar\O)$, and a Young measure  $\hat\nu:x\mapsto
\hat\nu_x$  on  $(\bar\O,\sigma)$ such that  $(\sigma,\hat\nu)$ is
attainable by this subsequence in the sense that $\forall g\!\in\!
C(\bar\O)\ \forall v_0\!\in\!\mathcal{S}$:
\begin{align}\label{basic}\lim_{k\to\infty}\int_\O g(x)v(Y_k(x))\md x =
\int_{\bar\O}\int_{\beta_\mathcal{S}\R^{M\times 
n}} g(x)v_0(A)\md\hat\nu_x(A)\md\sigma(x) \ , \end{align}
 where $v\in\us$.
 In particular,
putting $v_0:=v/(1+|\cdot|)=1$ in (\ref{basic}) we can see that
\be\label{measure} \lim_{k\to\infty}(1+|Y_k|)\ =\ \sigma \ \ \ \
\mbox{ weakly* in }\ \rca(\bar\O)\ . \ee If (\ref{basic}) holds,
we say that $\{y_k\}_{\in\N}$ generates $(\sigma,\hat\nu)$ and we denote the set of all
such pairs $(\sigma,\hat\nu)$ by $\DM$.

Comparing \eqref{measure} and \eqref{two} applied for $v(A):=|A|$ we immediately see that 
\begin{align}\label{measures}\sigma= \big(\int_{\R^{M\times N}}(1+|A|)\md\nu_x(A)\big)\mathcal{L}^N+\lambda\ .\end{align}
In particular, $\sigma_s=\lambda_s$, i.e., singular parts (in the Lebesgue decomposition) of both measures coincide.
In view of \cite[(A.7)]{ifmk} for all $g\in C(\bar\O)$ and all $v\in\ups$
 \begin{align}\label{basic-2}\lim_{k\to\infty}\int_\O
g(x)v( y_k(x))\md x
&=\int_\O\int_{\R^{M\times N}}g(x)v(A)\md\nu_x(A)\,\md x+
\int_{\bar\O}\int_{\beta_\mathcal{S}\R^{M\times N}\setminus\R^{M\times N}}g(x)\frac{v^\infty(A)}{1+|A|}\md\hat\nu_x(A)\,\md\sigma(x) \ .
\end{align}
Comparing this with \eqref{two}, we get 
\begin{align}\label{identity}
\int_{\bar\O}\int_{\beta_\mathcal{S}\R^{M\times N}\setminus\R^{M\times N}}g(x)\frac{v^\infty(A)}{1+|A|}\md\hat\nu_x(A)\,\md\sigma(x) =
\int_{\overline \Omega} \int_{S^{M\times N-1}}v^\infty(A)\nu^\infty_x(\md A) g(x)\,\md\lambda(x)\ .\end{align}
Setting $v^\infty(A):=|A|$ in the above identity \eqref{identity}  we get that 
\begin{align}\label{density}\frac{\md\lambda}{\md\sigma}(x)=\hat\nu_x(\beta_\mathcal{S}\R^{M\times N}\setminus\R^{M\times N})\ . \end{align}
Substituting back to \eqref{identity} we finally obtain  that for $\sigma$-a.a.~$x\in\bar\O$ and all continuous and positively one homogeneous functions $v^\infty$  
\begin{align}\label{infinity}
\int_{\beta_\mathcal{S}\R^{M\times N}\setminus\R^{M\times N}} \frac{v^\infty(A)}{1+|A|}\md\hat\nu_x(A)=\hat\nu_x(\beta_\mathcal{S}\R^{M\times N}\setminus\R^{M\times N})\int_{S^{M\times N-1}}v^\infty(A)\nu^\infty_x(\md A)\ .
\end{align} 
Comparing \eqref{basic} and \eqref{basic-2} we get for all $v\in\us$ and almost all $x\in\O$
\begin{align}\label{restriction}
\frac{\md\mathcal{L}^N}{\md\sigma}(x)\int_{\R^{M\times N}} v(A)\md\nu_x(A)=\int_{\R^{M\times N}}\frac{v(A)}{1+|A|}\md\hat\nu_x(A)\ ,
\end{align}
which connects the Young measure $\nu$ and the restriction of $\hat\nu$ on $\R^{M\times N}$.
Altogether , if $(\nu,\lambda,\nu^\infty)\in\genYM$ we construct $(\sigma,\hat\nu)\in\DM$ from \eqref{measures}, \eqref{density}, \eqref{infinity}, and from \eqref{restriction} and similarly, given  $(\sigma,\hat\nu)\in\DM$ we construct  $(\nu,\lambda,\nu^\infty)\in\genYM$.

%Taking into account \eqref{measures}, \eqref{density},  \eqref{infinity}, and \eqref{restriction} we arrive at the following statement.

%\begin{proposition}
%There is a one-to-one correspondence between $(\sigma,\hat\nu)$ defined in \eqref{basic} and $(\nu,\nu^\infty,\lambda)\in {\rm GYM}$.
%\end{proposition}
 %
 %
 %${\text DM}^{p}_{\mathcal R}(\Omega;\R^{M\times N})$ is the set of all DiPerna-Majda measures generated by bounded sequences in $L^{p}(\O;\R^{M\times N})$
 %and ${\text GDM}^{p}_{\mathcal R}(\Omega;\R^{M\times N})$ are those generated by gradients of bounded sequences in  $W^{1,p}(\O;\R^{M\times N}).$

\section{Quasiconvexity at the boundary versus JQCB}\label{sec:qcb-jqcb}

We recall the following notion originally introduced in \cite{bama}, in the form given in \cite{sprenger,mielke-sprenger}:
\begin{defn}[qcb: quasiconvex at the boundary]\label{def:qcb}
Given $v\in C(\RR^{M\times N})$, we say that $v$ is \emph{quasiconvex at the boundary (qcb) 
at the matrix $A\in\RR^{M \times N}$ with respect to the normal $\varrho\in S^{N-1}$} if there exists a vector $b=b(A)\in \RR^M$ such that
\begin{align}
  \int_{D_{\varrho}} [v(A+\nabla\varphi)-v(A)]\,dy\geq 
  \int_{B\cap \{y\mid y\cdot \varrho=0\}} b\cdot \varphi\,\md \cH^{N-1}(y)
  &~~\text{for every $\varphi\in W_0^{1,\infty}(B;\RR^M)$,}
  \label{qcb}
\end{align}
\end{defn}

\begin{remark}~\begin{enumerate}
\item[(i)]
The half-ball $D_\varrho=B\cap \{y|y\cdot \varrho<0\}$ is a prototypical example for a so-called ``standard boundary domain'' with respect to $\varrho$. Standard boundary domains are characterized by the fact that their boundary contains a planar part perpendicular to $\varrho$, and there, the test functions $\varphi$ remain unconstrained, while they have to vanish on the rest of the boundary. Quasiconvexity at the boundary can be equivalently defined using another such domain instead of $D_\varrho$ \cite{sprenger,mielke-sprenger}.
\item[(ii)] The boundary integral on the right hand side of \eqref{qcb} can be rewritten as a volume integral, using the Gauss Theorem and the boundary conditions for $\varphi$:
\begin{align}\label{qcbrhs}
  \int_{B\cap \{y\mid y\cdot \varrho=0\}} b\cdot \varphi\,\md \cH^{N-1}(y) = \int_{D_{\varrho}} (b\otimes \varrho):D\varphi\,\md y
\end{align}
\item[(iii)] Suppose that $v$ has a recession function $v^\infty$. Then, if $v$ is qcb at every $A\in \RR^{M\times N}$ with respect to $\varrho$, so is $v^\infty$.
\end{enumerate}
\end{remark}
The link to $JQCB(\varrho)$ as introduced in Definition~\ref{def:jqcb} is the following:
\begin{lemma}\label{lem:qcbJensen}
Let $\varrho\in S^{N-1}$, and let $v^\infty:\RR^{M\times N}\to\RR$ be continuous and positively $1$-homogeneous. If $v^\infty$ is qcb at $A$ with respect to $\varrho$ for every $A\in\RR^{M\times N}$, then $v^\infty\in JQCB(\varrho)$, i.e.,
\begin{align}
  v^\infty \Big(\int_{D_\varrho}\nabla\varphi\,\md y \Big)\leq \int_{D_\varrho} v^\infty(\nabla\varphi)\,\md y
  &\quad\text{for every $\varphi\in W_0^{1,\infty}(B;\RR^M)$.}
  \label{qcbJensen}
\end{align}
\end{lemma}

\begin{proof}
Let $Q=Q(\varrho)\subset \RR^N$ be an open unit cube centered at the origin, such that one of the sides of $Q$ is perpendicular to $\varrho$, and let 
\[
	\text{$Q_\varrho:=Q\cap H_\varrho$, 
	where $H_\varrho:=\{y\in \RR^N|y\cdot \varrho<0\}$}.
\]
$Q_\varrho$ is another example for a standard boundary domain with respect to $\varrho$, and it is not difficult to check (see \cite{sprenger,mielke-sprenger} for details) that for $v=v^\infty$, \eqref{qcb} is equivalent to
\begin{align}
  \int_{Q_{\varrho}} [v^\infty(A+\nabla\varphi)-v^\infty(A)]\,dy\geq 
  \int_{Q_{\varrho}} (b\otimes \varrho):\nabla\varphi\,\md y
  \quad\text{for every $\varphi\in W_0^{1,\infty}(Q;\RR^M)$,}
  \label{qcb2Q}
\end{align}
where we also used the analogue of \eqref{qcbrhs} on $Q_\varrho$.
Moreover, the class of test functions $\varphi$ in \eqref{qcb2Q} can be extended to $BV$ by a density argument. In particular, using the density of $W^{1,\infty}$ in $BV$ with respect to strict convergence we get
\begin{align}
\begin{aligned}
  \int_{H_\varrho} [\md v^\infty(A+D\varphi)(y)-v^\infty(A)\,\md y] \geq 
  \int_{H_\varrho} (b\otimes \varrho):\md D\varphi(y)&\\
  \text{for every $\varphi\in BV(\RR^N;\RR^M)$ with $\varphi=0$ on $\RR^N\setminus \bar Q$.}&
\end{aligned}  \label{qcb3Q}
\end{align}
Finally, since $\md v^\infty(\tilde{A}+D\varphi)(y)-v^\infty(\tilde{A})\,\md y=0$ on $\RR^N\setminus \supp \varphi$ for all $\tilde{A}\in\RR^{M\times N}$, we may replace the fixed matrix $A$ by any measurable function $F:\RR^N\to \RR^{M\times N}$ such that $F(y)=A$ for all $y\in H_\varrho\cap \supp \varphi$:
\begin{align}
\begin{aligned}
  \int_{H_\varrho} [\md v^\infty(F+D\varphi)(y)-v^\infty(F(y))\,\md y] \geq 
  \int_{H_\varrho} (b\otimes \varrho):\md D\varphi(y)&\\
  \text{for every $\varphi\in BV(\RR^N;\RR^M)$ with $\varphi=0$ on $\RR^N\setminus \bar Q$.}&
\end{aligned}   \label{qcb4Q}
\end{align}

Analogous to the equivalence of \eqref{qcb} and \eqref{qcb2Q}, our assertion \eqref{qcbJensen} holds if and only if for every $\psi\in W_0^{1,\infty}(Q;\RR^M)$,
\begin{align}
  v^\infty \Big(\int_{H_\varrho}\nabla\psi\,\md y \Big)\leq \int_{H_\varrho} v^\infty(\nabla\psi)\,\md y.
  \label{qcbJensenQ}
\end{align}
Therefore, given $\psi\in W_0^{1,\infty}(Q;\RR^M)$, let us see that \eqref{qcbJensenQ} indeed holds.

Due to the boundary conditions of $\psi$, $\int_{Q_{\varrho}} (\nabla\Psi(y))\varrho^\perp\,\md y=0$ for every vector $\varrho^\perp\in\RR^N$ perpendicular to $\varrho$. Therefore,
there exists a vector $a=a(\psi)\in \RR^M$ such that
\[
	a\otimes \varrho=\int_{Q_{\varrho}} \nabla\Psi(y)\,\md y\in \RR^{M\times N}. 
\]
We will use \eqref{qcb4Q} with the following functions: For $y\in Q_\varrho$ and $k\geq 2$, define
$F_k\in L^\infty(Q_\varrho;\RR^{M\times N})$ and $\varphi_k\in BV(H_\varrho;\RR^M)$ such that
\[
  F_k(y):=\left\{\begin{array}{ll}
  0~~&\text{for $-1<\varrho\cdot y< -\frac{1}{k}$,}\\
  k(a\otimes \varrho)~~&\text{for $-\frac{1}{k}\leq \varrho\cdot y<0$,}
  \end{array} \right.
  \quad
  \nabla \varphi_k(y):=-F_k(y)+\nabla \psi_k(y)~~\text{with}~\psi_k(y):=k^{N-1}\psi(ky),
\]
extended by zero to all of $\RR^N$. Here, notice that $F_k(y)=\nabla f_k(y)$ where $\nabla f_k$ is the absolutely continuous part of the $BV$ derivative $Df_k(y)$ of a potential $f_k\in BV(\RR^N;\RR^M)$ such that 
\begin{align*}
  &\text{$f_k=0$ on $\RR^N\setminus \overline{Q_\varrho}$}, \\
  &\text{$f_k$ is piecewise affine and continuous in $Q_\varrho$ with} \\
  &\text{$f_k=0$ on $\{y\in Q_\varrho\mid\varrho\cdot y<-\tfrac{1}{k}\}$ and}\\
  &\text{$\nabla f_k=k(a\otimes \varrho)$ on $\{y\in Q_\varrho\mid -\tfrac{1}{k}<\varrho\cdot y<0\}$. }
\end{align*}
In particular, $\varphi_k:=-f_k+\psi_k\in BV(\RR^N;\RR^M)$ is admissible as a test function in \eqref{qcb4Q}, and $F_k$ is constant on  $\supp \varphi_k\cap H_\varrho\subset (\frac{1}{k}\bar Q)\cap H_\varrho$ as required for \eqref{qcb4Q}.
Consequently, the singular part of $Df_k$ is supported on the jump set of $f_k$ given by 
\[
	J_k:=\big\{y\in \partial Q_\varrho~\big|~-\tfrac{1}{k}\leq \varrho\cdot y\leq 0\big\}\cup 
	(\partial Q_\varrho \cap \partial H_\varrho),
\] 
$f_k(\RR^N)\subset [0,a]$ (the closed one-dimensional line segment connecting $0$ and $a$ in $\RR^M$), and
\begin{align}\label{sidesgoaway}
  \abs{Df_k}(J_k\cap H_\varrho)\leq \abs{a} \cH^{N-1}(J_k\cap H_\varrho)=\abs{a}\frac{2N}{k} \to 0~~\text{as $k\to \infty$}.
\end{align}
Moreover,
\[
  \int_{H_\varrho}\md D\varphi_k(y)=-\int_{J_k\cap H_\varrho} \md Df_k(y)-\int_{H_\varrho} F_k\md y+\int_{H_\varrho} \nabla \psi_k\md y
  =0-a\otimes \varrho+a\otimes \varrho=0.
\]
Therefore, the right hand side of \eqref{qcb4Q} with $F=F_k$ and $\varphi=\varphi_k$ vanishes, and we obtain
\begin{align}
\begin{aligned}
  \int_{H_\varrho} [\md v^\infty(F_k+D\varphi_k)(y)-v^\infty(F_k(y))\,\md y] \geq 0.
\end{aligned}   \label{qcb4}
\end{align}
Using first the definition of $\varphi_k$ and then 
the definition of $F_k$, a change of variables $z=ky$ and the homogeneity of $v^\infty$, we infer that
\begin{align}
\begin{aligned}
  0 \leq & \int_{J_k\cap H_\varrho} \md v^\infty(Df_k)(y)+
  \int_{H_\varrho} [v^\infty(\nabla\psi_k(y))-v^\infty(F_k(y))]\,\md y \\
  = & \int_{J_k\cap H_\varrho} \md v^\infty(Df_k)(y)+
  \int_{H_\varrho} v^\infty(\nabla\psi(z))\,\md z - v^\infty(a\otimes \varrho).
\end{aligned}  
\end{align}
Since $\int_{Q_\varrho}\nabla\psi\,\md y= a\otimes \varrho$, due to \eqref{sidesgoaway}, passing to the limit as $k\to \infty$ we finally get \eqref{qcbJensenQ}. 
\end{proof}

\bigskip

\noindent
{\bf Acknowledgment.} This work was supported by GA\v{C}R through projects 14-15264S and 16-34894L. This research was  partly conducted  when MK held  the visiting Giovanni-Prodi professorship in the Institute of Mathematics, University of W\"{u}rzburg. Its   support and hospitality is gratefully acknowledged.

\end{document}

%% file: 13macros.tex
\usepackage[T1]{fontenc}
\usepackage{amsmath}                       
\usepackage{amssymb}                       
\usepackage{amsfonts}                      
\usepackage{amsthm}     
\usepackage{enumerate}                     % einfacheres {\"A}ndern von numerierten Aufz{\"a}hlungen       

\newcommand{\RR}{\mathbb R}
\newcommand{\NN}{\mathbb N}
\newcommand{\ZZ}{\mathbb Z}

\newcommand{\cL}{\mathcal{L}}

\newcommand{\cM}{\mathcal{M}}
\newcommand{\cH}{\mathcal{H}}

\newcommand{\cU}{\mathcal{U}}
\newcommand{\eps}{\varepsilon}

\newcommand{\supp}{\operatorname{supp}}

\newcommand{\To}{\longrightarrow}

\newcommand{\abs}[1]{\left\vert#1\right\vert}

\newcommand{\absn}[1]{\vert#1\vert}
\newcommand{\absb}[1]{\big\vert#1\big\vert}
\newcommand{\absB}[1]{\Big\vert#1\Big\vert}

\newcommand{\norm}[1]{\left\|#1\right\|}

\newcommand{\normn}[1]{\|#1\|}

\newcommand{\be}{\begin{equation}}
\newcommand{\ee}{\end{equation}}
\newcommand{\bald}{\begin{aligned}}
\newcommand{\eald}{\end{aligned}}

\newtheorem{theorem}{Theorem}%[]

\newtheorem{thm}[theorem]{Theorem}
\newtheorem{lem}[theorem]{Lemma}
\newtheorem{lemma}[theorem]{Lemma}

\newtheorem{prop}[theorem]{Proposition}
\newtheorem{proposition}[theorem]{Proposition}
\theoremstyle{definition}
\newtheorem{defn}[theorem]{Definition}
\newtheorem{definition}[theorem]{Definition}
\newtheorem{ex}[theorem]{Example}
\newtheorem{example}[theorem]{Example}
\theoremstyle{remark}
\newtheorem{rem}[theorem]{Remark}
\newtheorem{remark}[theorem]{Remark}
\renewcommand{\proofname}{\bfseries{Proof}}
\renewenvironment{proof}[1][\proofname]{\par
  \normalfont
  \trivlist
  \item[\hskip\labelsep\itshape
    \bfseries{#1.}]\ignorespaces
}{%
  \qed\endtrivlist
}